\documentclass[12pt, a4paper, reqno]{amsart}
\usepackage{amssymb}
\usepackage{amsthm}
\usepackage{tikz}
\usetikzlibrary{arrows,positioning} 
\usepackage{mathrsfs}
\usepackage{bbm}
\usepackage{ stmaryrd }

\usepackage{hyperref}
\hypersetup{
	colorlinks=true,
	linkcolor=blue,
	filecolor=magenta,  
	urlcolor=cyan,
}

\usepackage{hyperref}
\usepackage{cleveref}

\usepackage{amsfonts}
\usepackage{amsthm}
\usepackage{amsmath}
\usepackage{amscd}
\usepackage[latin2]{inputenc}
\usepackage{t1enc}
\usepackage[mathscr]{eucal}
\usepackage{indentfirst}
\usepackage{graphicx}
\usepackage{graphics}
\usepackage{epic}
\numberwithin{equation}{section}
\usepackage[margin=2.9cm]{geometry}
\usepackage{epstopdf} 
\usepackage{subfig}

\usepackage{multirow}
\usepackage{makecell}

\def\Aut{\operatorname{Aut}}
\def\End{\operatorname{End}}

\def\ker{\operatorname{ker}}

\def\dim{\operatorname{dim}}

\def\id{\operatorname{id}}

\def\Stab{\operatorname{Stab}}

\def\Der{\operatorname{Der}}

\def\Spec{\operatorname{Spec}}
\def\Hom{\operatorname{Hom}}

\def\C{\mathbb{C}}

\def\R{\mathbb{R}}
\def\N{\mathbb{N}}
\def\Z{\mathbb{Z}}

\def\TT{\mathcal{T}}
\def\LL{\mathcal{L}}
\def\OO{\mathcal{O}}
\def\KK{\mathcal{K}}
\def\BB{\mathcal{B}}
\def\FF{\mathcal{F}}
\def\JJ{\mathcal{J}}

\def\MM{\mathcal{M}}

\def\NN{\mathcal{N}}
\def\PP{\mathcal{P}}

\def\II{\mathcal{I}}

\def\a{\mathfrak{a}}
\def\b{\mathfrak{b}}

\def\m{\mathfrak{m}}

\def\n{\mathfrak{n}}
\def\p{\mathfrak{p}}
\def\q{\mathfrak{q}}
\def\g{\mathfrak{g}}

\def\h{\mathfrak{h}}
\def\r{\mathfrak{r}}
\def\k{\mathfrak{k}}
\def\l{\mathfrak{l}}
\def\s{\mathfrak{s}}
\def\o{\mathfrak{o}}

\def\d{\partial}

\def\ol{\overline}

\def\Ext{\text{Ext}}

\def\ul{\underline}

\def\sub{\subseteq}

\newtheorem{thm}{Theorem}[section]
\newtheorem{cor}[thm]{Corollary}
\newtheorem{lemma}[thm]{Lemma}
\newtheorem{prop}[thm]{Proposition}

\theoremstyle{definition}
\newtheorem{definition}[thm]{Definition}
\theoremstyle{remark}
\newtheorem{remark}[thm]{Remark}
\newtheorem{example}[thm]{Example}

\numberwithin{equation}{section}

\usepackage{relsize}
\usepackage{fancyhdr}
\usepackage[all,cmtip]{xy}

\begin{document}
	
	\title{Spherical Supervarieties}
	
	\author[Alexander Sherman]{Alexander Sherman}
	
	\maketitle
	\pagestyle{plain}

	\begin{abstract}  We give a definition of the notion of spherical varieties in the world of complex supervarieties with actions of algebraic supergroups.  A characterization of affine spherical supervarieties is given which generalizes a characterization in the classical case.  We also explain some general properties of the monoid of highest weights.  Several examples are given that are interesting in their own right and highlight differences with the classical case, including the regular representation, symmetric supervarieties, and actions of split supergroups.
	\end{abstract}

\section{Introduction}\label{sec_intro}

Let $G$ be a complex algebraic supergroup with $G_{0}$ reductive, where $G_{0}$ is the even underlying algebraic group of $G$.  We call such supergroups quasireductive.  We would like to consider supervarieties with actions of such supergroups which have an especially large amount of symmetry; namely, we would like a hyperborel subgroup (see \cref{defn_borel_group}) to have an open orbit.  For those familiar with Lie superalgebras, the notion of hyperborel subalgebra agrees with the usual notion of Borel subalgebra for many heavily studied cases, apart from queer superalgebras (see \cref{borel_rmk}).  We call such supervarietes spherical, generalizing the classical notion to the super world.

In the classical world, spherical varieties are a highly rich and well-studied class of varieties which simultaneously generalizes toric varieties, flag varieties, and symmetric spaces.  They provide connections between representation theory, combinatorics, and algebraic geometry.   Affine spherical varieties also have a close relationship with multiplicity-free spaces in symplectic geometry, and were used by F. Knop and I. Losev to prove Delzant's conjecture (\cite{losev2009proof}).  In work spanning several decades up to the mid-2010s, Bravi, Brion, Cupit-Foutou, Knop, Losev, Luna, Pezzini, Vust, and others completed the combinatorial classification of all spherical varieties.  

It is interesting to ask how spherical varieties generalize to the super world.   Classically, the first theorem giving a connection to representation theory states that an affine $G$-variety $X$ is spherical if and only if $\C[X]$ is multiplicity-free as a $G$-module.  In the super case complete reducibility is rare and so such a statement is too much to hope for.  It is seen in this article that $\C[X]$ may not be completely reducible for an affine spherical supervariety $X$; however the socle of $\C[X]$ must be multiplicity-free. 

On the flip side, and perhaps more surprising superficially, there are situations in which a $G$-supervariety $X$ is affine, $\C[X]$ is completely reducible and multiplicity-free, but $X$ is not spherical.  Thus this connection does not generalize nicely to the super world.  However, we find a characterization of sphericity in terms of the subalgebra of $\C[X]$ generated by $B$-highest weight functions, where $B$ is a hyperborel subgroup	 (\cref{spher_affine_char}).  This characterization generalizes the classical fact that an affine $G$-variety $X$ is spherical if and only if $X//N$ is a toric variety for a maximal torus $T$ of a Borel subgroup $B$, where $N$ is the unipotent radical of $B$.

The author began studying examples of spherical varieties in \cite{sherman2020spherical}.  In that work, indecomposable spherical representations were found for a large class of quasireductive groups, and the structure of the algebra of functions was determined.  That paper and this one seek to understand affine $G$-supervarieties better, in particular in understanding how the geometry of the action is connected to the representation theory of the space of functions.  

This work has been in progress for several years now by other authors within the study of symmetric superspaces.  In \cite{sahi2016capelli} and \cite{sahi2018capelli} the Capelli eigenvalue problem has been studied for supersymmetric pairs coming from simple Jordan superalgebras.  In  \cite{alldridge2012harish} a generalization of the Harish-Chandra isomorphism theorem was given, and in \cite{alldridge2015spherical} certain facts about the socle of the space of functions is proven, amongst other things.  Further, in \cite{sergeev2004deformed} the combinatorics of root systems gotten from supersymmetric pairs is used to construct integrable systems.  We hope further insight can be gained through the more general lens of spherical supervarieties.

\subsection{Structure of paper} We begin with definitions and explanations about notation for supervarieties in section 2.  In section 3 we define actions by supergroups and state some lemmas about actions we plan to use later on, and in section 4 we define quasireductive supergroups and the notion of hyperborel subgroup.  In section 5 we define spherical supervarieties and give characterization theorems and some consequences.  Then in section 6 we discuss several examples the author has considered, with new results stated.  The first appendix briefly looks at the notion of spherical actions of quasireductive Lie superalgebras on supervarieties by vector fields.  Finally, the second appendix addresses some generalities about smoothness of supervarieties.

\subsection{Acknowledgments}  The author would like to thank his advisor, Vera Serganova, for stimulating their interest in spherical supervarieties, and useful discussions along the way.  The author also thanks Alexander Alldridge for many insightful explanations and discussions about supergeometry, supergroups, and their actions.

\section{Supergeometry}

We are work in the algebraic setting.  For the basic definitions on superschemes, see chapter 10 of \cite{carmeli2011mathematical}.  Although we work over the complex numbers, one could just as well work over any algebraically closed field of characteristic zero.

\subsection{Notation and terminology} For a super vector space $V$ we write $V=V_{\ol{0}}\oplus V_{\ol{1}}$ for its parity decomposition, and $\Pi V=\C^{0|1}\otimes V$ for the parity shift, where $\C^{0|1}=0\oplus\C$.  Given a homogeneous element $v\in V$, we write $\ol{v}\in\Z/2\Z=\{\ol{0},\ol{1}\}$ for its parity.

For a superscheme $X$, write $|X|$ for the underlying topological space of $X$.  Let $\OO_X$ denote its structure sheaf and $\OO_X=(\OO_{X})_{\ol{0}}\oplus(\OO_{X})_{\ol{1}}$ for the parity decomposition of this sheaf.  For a point $x\in|X|$, we write $\OO_{X,x}$ for the stalk of the sheaf $\OO_X$ at $x$ which will be a local superalgebra, and $\m_x$ for its unique maximal ideal.  For a superalgebra $R$, we write $X(R)$ for the $\Spec R$-points of $X$.  In particular, we write $X(\C)$ for the complex points of $X$, which for the spaces we consider will be exactly the closed points.  For an open subset $|U|\sub|X|$, we write $U$ for the superscheme obtained by restriction of $X=(|X|,\OO_X)$ to $|U|$, and we call $U$ an open subscheme of $X$.

We write $X_{0}$ for the even subvariety of a superscheme $X$, that is the space cut out by the ideal sheaf $\JJ_X$ generated by $(\OO_X)_{\ol{1}}$.  Write $i_{X}:X_{0}\hookrightarrow X$ for the corresponding closed embedding, or sometimes simply $i$ if the space is clear from context.  Let $\NN_X:=\JJ_X/\JJ_X^2$ be the conormal sheaf, which is a quasicoherent sheaf on $X_{0}$. 

For a superscheme $X$ such that $|X|$ is Noetherian and irreducible, write $\C(X)$ for the stalk of $\OO_{X}$ at the generic point of $|X|$.  Then for any open subscheme $U$ of $X$ we have a natural map $\Gamma(|U|,\OO_X)\to\C(X)$.  This map may not be injective (although for us it always will be), but if $f$ is a section over $|U|$ we will sometimes speak of it as an element of $\C(X)$ with the understanding that we are talking about its image under this restriction map.

Quasi-coherent, coherent sheaves, and line bundles on a superscheme may be defined in the same way as for schemes.  Further we may take direct sums and tensor products between them to obtain new (quasi)-coherent sheaves.  For a quasicoherent sheaf $\FF$ on a superscheme $X$, and a point $x\in|X|$, we write $\FF_x$ for its stalk at $x$.  We say that $\FF$ is globally generated if the natural map
\[
\Gamma(X,\FF)\otimes_{\C}\OO_X\to \FF
\]
is a surjective map of sheaves.  Similarly, we say a subspace $V\sub\Gamma(X,\FF)$ generates $\FF$ if the morphism
\[
V\otimes_{\C}\OO_X\to \FF
\]
is a surjective map of sheaves.

\subsection{Supervarieties}

\begin{definition}
	We define a supervariety to be an irreducible superscheme $X$ over $\C$ such that the following conditions are satisfied:
	\begin{enumerate}
		\item $X$ admits a finite cover by affine open subschemes of the form $\Spec A$, where $A$ is a finitely-generated superalgebra over $\C$;
		\item For any open subscheme $U\sub X$, the map $\Gamma(U,\OO_X)\to \C(X)$ is injective;
		\item The superalgebra $\C(X)$ is an integral superdomain, that is the zero divisors of $\C(X)$ are exactly $(\C(X)_{\ol{1}})$;
		\item $X$ is separated over $\C$, that is the the diagonal morphism $X\to X\times X$ is a closed embedding.
	\end{enumerate}
\end{definition}

\begin{remark}
	\begin{itemize}
		\item If $X$ is a supervariety, then for all open sets $|U|,|U'|\sub |X|$ with $|U'|\sub |U|$, the restriction map $\Gamma(|U|,\OO_X)\to\Gamma(|U'|,\OO_X)$ is injective.  This follows from functoriality of restriction.
		\item If $A$ is a finitely generated $\C$-algebra such that $\Spec A$ is a supervariety, condition (2) implies that the zero divisors of $A_{\ol{0}}$ are all nilpotent. 
		\item Closures of supervarieties are supervarieties, as are orbit closures of actions by supergroups.
		\item If $X$ is a supervariety, $X_0$ need not be a variety in the usual sense in that it may not be integral.  A natural example arises from the action of the supergroup $G=GL(1|2)$ on $S^2\C^{1|2}$.  The orbit closure of an eigenvector for $G_0$ is a supervariety, but the underlying scheme has nilpotent functions and is not a variety.  Nevertheless, $X_0$ will be integral on an open subscheme.
	\end{itemize}
	
\end{remark}

\subsection{Quasiprojective supervarieties}
\begin{definition}
\begin{enumerate}
	\item An affine supervariety is a supervariety of the form $\Spec A$ for a superalgebra $A$.
	\item Given an $\N$-graded superalgebra $A=\bigoplus\limits_{n\in\N} A_n$, one may define $\operatorname{Proj}A$ as is done in \cite{carmeli2011mathematical}, following \cite{hartshorne2013algebraic}.  We say a supervariety is projective if it is of the form $\operatorname{Proj}A$ where $A_0=\C$ and $A_1$ is finite-dimensional.  A supervariety is quasiprojective if it is a dense open subscheme of a projective supervariety.
\end{enumerate}	
\end{definition}

\begin{remark}
	Quasi-projective supervarieties are not as pervasive as quasiprojective varieties in the category of varieties.  Indeed, there are many linear algebraic supervarieties of importance which are not projective, including `most' super Grassmanians.  (See \cite{manin2013gauge} for a discussion in the analytic setting,  or \cite{carmeli2011mathematical} for the algebraic setting.)  This is a significant hindrance in understanding such spaces and how supergroups act on them.
\end{remark}

Projective superspace $\mathbb{P}^{m|n}$ and very ample line bundles may be defined as usual (see \cite{carmeli2011mathematical}).  As in the classical setting, we have the following two important results for quasiprojective supervarieties.

\begin{lemma}
	Let $X$ be a quasiprojective supervariety and $\LL$ a very ample line bundle on $X$.  Then for every homogeneous $f\in \C(X)$, there exists $n\geq0$ and homogeneous sections $s_1,s_2\in\Gamma(X,\LL^{\otimes n})$ such that $f=s_1/s_2$.
\end{lemma}

\begin{prop}\label{serre_thm}
	Let $X$ be a quasiprojective supervariety, $\LL$ a very ample line bundle on $X$, and $\FF$ a coherent sheaf on $X$.  Then for some $n\geq0$, $\FF\otimes_{\OO_X}\LL^{\otimes n}$ is globally generated.
\end{prop}

\subsection{Graded supervarieties}

Let $X_0$ be a scheme and $\NN$ a coherent sheaf on $X_0$.  Then $X=(|X_0|,\Lambda^\bullet\NN)$ is a superscheme in a natural way.  
\begin{definition}\label{graded_defn}
	We say that a supervariety $X$ is graded if there exists a coherent sheaf $\NN$ on $X_0$ and an isomorphism $X\cong(|X|,\Lambda^\bullet\NN)$.  We call an isomorphism of $X$ with $(|X|,\Lambda^\bullet\NN)$ a grading of $X$.
	
	We say $X$ is locally graded if it admits a covering by open subschemes that are graded.
\end{definition}  

\begin{remark}
	\begin{itemize}\item If $X$ is graded and $X\cong(|X|,\Lambda^\bullet\NN)$, then $\NN_X\cong\NN$.  
		\item Smooth affine supervarieties are always graded (see \cite{voronov1990elements}).	Thus smooth supervarieties are always locally graded.		 
		\item Using the cohomology argument given in \cite{voronov1990elements}, one can show that the property of being locally graded is affine local.  Thus locally graded affine supervarieties are graded.
	\end{itemize}
\end{remark}

\begin{remark}[Caution]
	The term graded is sometimes used instead of super (see for instance \cite{koszul1982graded}).  Another term that has been used to mean graded is split, as in \cite{vishnyakova2011complex}.  However others (e.g.\ \cite{voronov1990elements}) have used split to mean there is a splitting of the surjective morphism $\OO_X\to\OO_{X_0}$.  We will use the term graded and hope no confusion will arise.
\end{remark}

When $X$ is graded, so that $X\cong(|X|,\Lambda^\bullet\NN_X)$, its structure sheaf becomes endowed with a $\Z$-grading according to the exterior powers of the conormal sheaf, namely $(\Lambda^\bullet\NN_X)_i=\Lambda^i\NN_X$.  However a graded supervariety $X$ has, in general, many isomorphisms with $(|X|,\Lambda^\bullet\NN)$ (see for instance \cite{koszul1994connections}).

\subsection{Tangent sheaf and tangent spaces}
\begin{definition}
	For a supervariety $X$, define the tangent sheaf $\TT_X$ as the unique sheaf defined on any affine open subscheme $U=\Spec A$ of $X$ by $\Gamma(|U|,\TT_X)=\Der(A)$, that is all (not necessarily even) $\C$-linear algebra derivations of $A$.  In this way $\TT_X$ is a coherent sheaf of Lie superalgebras on $X$, and $\Gamma(U,\TT_X)$ acts by super derivations on $\Gamma(U,\OO_X)$.
\end{definition}

\begin{definition}
	Given $x\in X(\C)$, we define the tangent space at $x$ to be the super vector space $T_xX$ given by point derivations $\delta:\OO_{X,x}\to\C$, i.e.\ maps of vector spaces such that $\delta(fg)=\delta(f)g(x)+(-1)^{\ol{\delta}\ol{f}}f(x)\delta(g)$.  Note that the minus sign is not strictly necessary since if $\ol{f}=\ol{1}$ then $f(x)=0$.
\end{definition}

\begin{remark}\label{smooth_rmk}
	We have a natural identification $T_xX\cong (\m_x/\m_x^2)^*$.  Further, there is a natural map of super vector spaces
	\[
	\TT_{X,x}\to T_xX
	\]
	given by $D\mapsto (f\mapsto D(f)(x))$.  This map is not always surjective.  We say that $X$ is smooth at $x$ if it is surjective.  See \cref{app_smooth} for a discussion of smoothness of superschemes.
\end{remark}

\section{Supergroups and their Actions}  

\subsection{Supergroups} See sections 8, 9, and 11 of \cite{carmeli2011mathematical} for more on the foundations of (algebraic) supergroups and their actions.

\begin{definition}
	An algebraic supergroup is a complex supervariety $G$ equipped with morphisms $m=m_G:G\times G\to G$, $s=s_G:G\to G$, and $e=e_G:\Spec\C\to G$ satisfying the usual commutativity conditions:
	\[
	m\circ(m\times \id_G)=m\circ(\id_G\times m), 
	\]
	\[
	m\circ (e\times \id_G)=m\circ(\id_G\times e)=\id_G,
	\]
	and
	\[
	m\circ (\id_G\times s)\circ\Delta_G=m\circ (s\times\id_G)\circ\Delta_G=e.
	\]
	where $\Delta_G:G\to G\times G$ is the diagonal embedding.  In addition, we assume throughout this article that $G$ is linear, that is affine.  
\end{definition}

\begin{definition} For $u_e\in T_eG$, construct a right-invariant vector field $u_L$ on $G$ via left infinitesimal translation by the equation
	\[
	u_L(f)=-(u_e\otimes 1)(m^*(f))
	\]
	Then the value of $u_L$ at $e$ as a tangent vector is $-u_e$.  Write $\g=\operatorname{Lie}G$ for the Lie superalgebra of right-invariant vector fields on $G$.  The restriction map $\g\to T_eG$ is an isomorphism of super vector spaces, so we will freely identify $\g$ with $T_eG$ when convenient.  Given $u_e\in T_eG$ we may also construct a left-invariant vector field on $G$ via right infinitesimal translation given by
	\[
	u_R(f)=(1\otimes u_e)(m^*(f)).
	\]
	The Lie superalgebra of left-invariant vector fields is canonically isomorphic to the lie superalgebra of right vector fields via $u_L\mapsto u_R$.  
\end{definition}

\begin{remark}
	If $G$ is an algebraic supergroup, then $G_0$ is an algebraic group in the usual sense, and we have a canonical isomorphism $\g_{\ol{0}}\cong\operatorname{Lie}(G_{0})$.
\end{remark}
\subsection{Actions}
\begin{definition}\label{def_action}
	Let $X$ be a supervariety and $G$ an algebraic supergroup.  An action of $G$ on $X$ is a morphism $a:G\times X\to X$ such that 
	\[
	a\circ(m_G\times\id_X)=a\circ(\id_G\times a)
	\]
	and
	\[
	a\circ (e\times \id_X)=\id_X
	\]
\end{definition}

Given an action of $G$ on $X$, we obtain a homomorphism $\rho_a:\g\to\Gamma(X,\TT_X)$ as follows.  For an open set $U\sub X$, choose an open subset $U'\sub G$ containing the identity such that $a$ sends $U'\times U$ into $U$.  Let $f\in\Gamma(U,\OO_X)$ and $u\in\g$.  Then define the action of $u$ on $f$ by
\[
u(f)=-(u_e\otimes 1)(a^*(f)).
\]
The map $\rho_a$ determines an action of the Lie superalgebra $\g$ on $X$ (see \cref{app_superalg_actions}).

\begin{remark}
	If a Lie supergroup $G$ acts on a supervariety $X$, then by functoriality $G_0$ acts on both $X$ and $X_0$ in a natural way.
\end{remark}

We omit the proof of the following result.  It can be shown by developing the notion of an action of a super Harish-Chandra pair, and showing it is equivalent to an action of the corresponding supergroup.  The result is stated for supermanifolds without proof in \cite{deligne1999notes} and a full proof for supermanifolds is given in section 4.5 of \cite{balduzzi2011supermanifolds}.   The author provided a full proof for the algebraic case in their PhD thesis.

\begin{thm}\label{HC_pair_action}
	Let $G$ be a Lie supergroup with $\g=\operatorname{Lie}(G)$, and suppose that $X$ is a supervariety.  Suppose that $G_{0}$ acts on $X$ via $a_0:G_0\times X\to X$, and that we have a homomorphism of Lie superalgebras $\rho:\g\to\Gamma(X,\TT_X)$ such that
	\begin{enumerate}
		\item $\rho|_{\g_{\ol{0}}}(u)=-(u\otimes 1)\circ a_0^*$ for all $u\in\g_{\ol{0}}$;
		\item $\rho(\operatorname{Ad}(g)(u))=(a_0^{g^{-1}})^*\circ\rho(u)\circ(a_0^g)^*$ for all $g\in G_{0}$ and $u\in\g$, where $a_0^g=a_0\circ(i_g\times \id_X)$, where $i_g:\{g\}\to G_0$ is the natural inclusion.
	\end{enumerate}
	Then there exists a unique action $a:G\times X\to X$ of $G$ on $X$ such that $a|_{G_0}=a_0$ and $\rho_a=\rho$. 
\end{thm}

We will often use this result in the form of the following corollary:

\begin{cor}\label{stability}
	Suppose that a Lie supergroup $G$ acts on a supervariety $X$, and that the open subset $|U|\sub|X|$ is stable under the action of $G_0$.  Then the open subvariety $U$ is stable under the action of $G$, i.e.\ the action of $G$ on $X$ restricts to an action of $G$ on $U$.
\end{cor}

\subsection{Orbit maps and stabilizers} For $x\in X(\C)$, we have an orbit map at $x$, $a_x:G\to X$, given by $a\circ (\id_G\times i_x)$, where $i_x:\{x\}\to X$ is the natural inclusion.  We refer to $a_x^{-1}(x)$, the fiber of this morphism over $x$, as the stabilizer $\Stab_G(x)$ of $x$, a closed subgroup of $G$ (see section 11.8 of \cite{carmeli2011mathematical}.  The following lemma is well-known (see e.g.\ Lemma 4 of \cite{vishnyakova2011complex}).

\begin{lemma}\label{orbit_diff} For $x\in X(\C)$, the differential of the orbit map $a_x$ at the identity of $G$, $(da_x)_e:T_eG\to T_xX$, coincides with the natural evaluation map $\rho_a(\g)\to T_xX$.\end{lemma}

The Lie superalgebra of $\Stab_G(x)$ is then exactly the kernel of the restriction morphism $\rho_a(\g)\to T_xX$.

\begin{definition}
	Suppose that $G$ acts on $X$.  We say that the action is a submersion at a point $x\in X(\C)$ if the map $a_x:G\to X$ is a submersion at $e_G\in G(\C)$ (or equivalently at any point of $G$).  In this case, the locus of points where the map is a submersion will be an open subset of $|X|$, and we refer to the open subvariety defined by this locus as an open orbit of $G$.   If all of $X$ is an open orbit of $G$, we say that $X$ is a homogeneous $G$-supervariety.
\end{definition}
\begin{remark}
	An action is a submersion at $x$ if and only if the evaluation map $\rho_a(\g)\to T_xX$ is surjective by \cref{orbit_diff}.  By \cref{smoothness}, an open orbit of $G$ must be smooth. 
\end{remark}

\begin{prop}\label{submersion_injec}
	Let $X$ be a supervariety, and let $a:G\times X\to X$ be an action of an algebraic supergroup $G$ on $X$.  Then for $x\in X(\C)$ the following are equivalent:
	\begin{enumerate}
		\item $a_x$ is a submersion;
		\item the pullback morphism of sheaves $a_x^*:\OO_X\to(a_x)_*\OO_G$ is injective;
		\item there exists a line bundle $\LL$ such that the pullback morphism $a_x^*:\LL\to (a_x)_*a_x^*\LL$ is injective.
		\item for all line bundles $\LL$ on $X$, the pullback morphism $a_x^*:\LL\to (a_x)_*a_x^*\LL$ is injective.
	\end{enumerate}
\end{prop}

\begin{proof}
	We first prove the equivalence $(1)\iff(2)$.  Let $K$ be the stabilizer of $x$, and write $\pi:G\to G/K$ for the natural projection. Then the natural map of sheaves $\OO_{G/K}\to\pi_*\OO_{G}$ is injective.  There is an induced $G$-equivariant immersion $b_x:G/K\to X$ and this map factors the orbit map $a_x$.  Therefore if $a_x$ is a submersion, $G/K\to X$ is too, and hence it induces an isomorphism of $G/K$ onto an open subset of $X$.  By our assumption that restriction of functions is injective on supervarieties, the map 
	\[
	\OO_X\to b_*\OO_{G/K}\to b_*\pi_*\OO_{G}=(a_x)_*\OO_{G}
	\] 
	is injective.
	
	If $a_x$ is not a submersion, first suppose that the underlying image of $G/K$ in $X$ is not open.  Then we may choose a non-nilpotent function on $X$ which vanishes on the underlying closed subscheme defined by its image, so that some power of this function will vanish under pullback, and thus $a_x^*$ is not injective.  Therefore assume $G/K$ has an underlying open image, say $|U|\sub|X|$.  Then we may restrict to the open subscheme $U$ of $X$, and there the morphism $G/K\to U$ will be an isomorphism on closed points and an immersion, but not a submersion.  One may then show that this map is a closed embedding, by considering the map on local rings and using Nakayama's lemma.  Hence the map on stalks is surjective, and so if it were also injective the map would be an isomorphism on this open set $U$, contradicting the fact that $a_x$ is not submersive.
	
	Now we show that $(2)\iff(3)\iff(4)$.  For any line bundle $\LL$ on $X$, we may cover $X$ with open sets $U_i$ for which $\LL|_{U_i}\cong\OO_{U_i}$.  Then by functoriality, over $U_i$ the pullback morphism $\LL\to(a_x)_*(a_x)^*\LL$ is identified with $\OO_{U_i}\to (a_x)_*(a_x)^*\OO_{U_i}\cong (a_x)_*\OO_{U_i}$, and one is injective if and only if the other is.  Since injectivity of a morphism of sheaves is a local property, we are done.
	
\end{proof}

\begin{definition}
Let $G$ be a supergroup, $X$ a $G$-supervariety, and $\FF$ a quasicoherent sheaf on $X$.  Write $a:G\times X\to X$ for the action morphism and $p:G\times X\to X$ for the natural projection.  A $G$-linearization of $\FF$ is a choice of isomorphism of $\OO_{G\times X}$-modules $\varphi_{\FF}:a^*\FF\cong p^*\FF$ such that the following cocycle condition is satisfied: 
\[
(m_G\times \id_{X})^*\varphi_\FF=p_{23}^*\varphi_\FF\circ (\id_{G}\times a)^*\varphi_\FF,
\]
where $p_{23}:G\times G\times X\to G\times X$ is the projection onto the second and third factor. 

We will also refer to a linearized quasicoherent sheaf as an equivariant quasicoherent sheaf.  A morphism of $G$-equivariant quasicoherent sheaves is a morphism of quasicoherent sheaves that respects their equivariant structure.  Further, the pullback and pushforward of an equivariant sheaf along an equivariant map admits a natural equivariant structure.
\end{definition}
From a linearization on a quasicoherent sheaf $\FF$, one obtains the structure of a $G$-module on $\Gamma(X,\FF)$ and the structure of a $\g$-module on rational sections of $\FF$.  The latter structure satisfies the following Liebniz property: for $u\in\g$, $s$ a section of $\FF$, and $f$ a section of $\OO_X$, we have 
\[
u(fs)=u(f)s+(-1)^{\ol{f}\ol{u}}fu(s).
\]
\begin{definition}
	If $\g$ acts on a supervariety $X$, and $\FF$ is a quasicoherent sheaf on $X$, then a $\g$-linearization of $\FF$ is a Lie superalgebra homomorphism $\g\to\End(\FF)$ satisfying the above Liebniz rule.
\end{definition}
If $\LL$ is an equivariant line bundle, then if we choose a trivialization $\psi:\LL|_U\to\OO_X|_U$ over an open subscheme $U$, we obtain an action of $\g$ on sections of $\OO_X$ over $U$, which we write as $u_{\psi}$.  For $u\in\g$ and $f$ a section of $\OO_X$ over $U$, the action satisfies
\[
u_\psi(f)=u(f)+(-1)^{\ol{u}\ol{f}}fu_{\psi}(1).
\]

\begin{remark}  It is not true that every line bundle has some power that admits a linearization from a group action, even on a smooth supervariety.  An example is given by $\mathbb{P}^{1|n}$ for $n\geq 2$, whose Picard group is $\Z\oplus\C^{2^{n-2}(n-2)+1}$ (see \cite{cacciatori2018projective}), but only the bundles from the $\Z$ factor have a linearization.
	
	Because of this we will often need an additional assumption that there exists an equivariant very ample line bundle.  The author does not know whether such a line bundle must exist on a quasiprojective supervariety.
\end{remark}
\begin{prop}\label{homog_line_bundle_char}
	Let $X$ be a $G$-supervariety.  If $X$ is homogeneous then for any $G$-equivariant line bundle $\LL$, a non-zero $G$-submodule $V$ of $\Gamma(X,\LL)$ generates $\LL$.  Conversely, if $X$ is quasiprojective and admits a very ample $G$-equivariant line bundle $\LL$, then the converse also holds.
\end{prop}

\begin{proof}
	Let $\LL$ be a $G$-equivariant line bundle on $X$, and let $V\sub\Gamma(X,\LL)$ be a nonzero $G$-submodule.  Then if $V$ does not generate $\LL$, there must exist a point $x\in X(\C)$ such that when we pass to the stalk of $\LL$ at $x$ we find that $V\sub\m_x\LL_x$.  Therefore there exists a maximal positive integer $n$ such that $V\sub\m_x^n\LL_x$.  
	
	Let $s\in V$ be in $\m_x^n\LL_x\setminus\m_x^{n+1}\LL_x$.  Choose a trivialization $\psi:\LL_x\cong\OO_x$ and write $f=\psi(s)$ so that $f\in \m_x^{n}\setminus\m_x^{n+1}$.  Then because $X$ is homogeneous there exists $u\in\g$ such that $u(f)\in\m_x^{n-1}\setminus\m_x^{n}$.  Therefore
	\[
	u_{\psi}(f)=u(f)+(-1)^{\ol{f}\ol{u}}fu_{\psi}(1)\in\m_x^{n-1}\setminus\m_x^n,
	\]
	so in particular $u(s)\in\m_{x}^{n-1}\LL_x\setminus\m_x^n\LL_x$ and $u(s)\in V$, a contradiction.  Therefore instead $V$ must generate $\LL$.
	
		Let $X$ be quasiprojective admitting a very ample $G$-equivariant line bundle $\LL$, and suppose that $X$ is not homogeneous.  Then there exists  $x\in X(\C)$ such that $a_x$ is not a submersion. By \cref{submersion_injec},  the pullback morphism $\OO_X\to (a_x)_*\OO_X$ is not injective.  Write $\KK$ for its kernel, so that we obtain an exact sequence of $G$-equivariant sheaves
	\[
	0\to\KK\to\OO_X\to (a_x)_*\OO_X.
	\]
	Since $\LL$ is $G$-equivariant and flat as an $\OO_X$-module, we may twist by it and obtain for each $n\in\N$  an exact sequence of $G$-equivariant sheaves
	\[
	0\to\KK(n)\to\LL^{\otimes n}\to (a_x)_*\OO_X\otimes\LL^{\otimes n},
	\] 
	where we write $\KK(n):=\KK\otimes\LL^{\otimes n}$.  Since $\LL$ is very ample, by \cref{serre_thm} there exists $n>0$ such that $\KK(n)$ is globally generated, and so in particular $\Gamma(X,\KK(n))\neq0$.  However by left exactness of global sections, $\Gamma(X,\KK(n))\sub\Gamma(X,\LL^{\otimes n})$ is a non-zero $G$-submodule.  Since the morphism $\OO_X\to (a_x)_*\OO_x$ is not zero, necessarily $\Gamma(X,\KK(n))$ cannot generate $\LL^{\otimes n}$, a contradiction, and we are done.
\end{proof}

\begin{remark}\label{algebra_action_remark}
	This proof shows that if a Lie superalgebra $\g$ acts homogeneously on a supervariety $X$ (see \cref{app_superalg_actions} for the meaning of this) then if $\LL$ is a $\g$-equivariant line bundle and $V$ is a non-zero $\g$-stable submodule of $\Gamma(X,\LL)$, $V$ must generate $\LL$.
\end{remark}

\begin{cor}\label{affine_stable_ideals}
If $X$ is an affine $G$-supervariety and $\LL$ is a $G$-equivariant line bundle on $X$, then  $\Gamma(X,\LL)$ admits a non-zero $G$-stable $\Gamma(X,\OO_X)$-submodule if and only if $X$ is not homogeneous.  In particular, $\C[X]$ has no nontrivial $G$-stable ideals if and only if $X$ is homogeneous.
\end{cor}
\begin{proof}
We apply \cref{homog_line_bundle_char}, using that if $X$ is affine the global sections functor is exact.
\end{proof}

\subsection{Rational invariants} In the classical world, if an algebraic group $G$ acts on a space $X$, then it admits an open orbit if and only if $\C(X)^\g=\C$.  In the super world, this general principle no longer holds.
\begin{example}
	Consider the action of $GL(0|n)$ on $X=\C^{0|n}$ by the standard representation of $GL(0|n)$.  This supervariety has one point, and the orbit of that point is just itself, so there is not an open orbit.  We have $\C(X)=\Lambda^\bullet(\C^n)^*$, and this is a multiplicity-free representation of $\g=\g\l(n)$, so in particular $\C(X)^\g=\C$.
\end{example}
We do have the forward direction:
\begin{prop}
	If a Lie supergroup $G$ acts on a supervariety $X$ with an open orbit, we have $\C(X)^\g=\C$.
\end{prop}
\begin{proof}
	Let $f\in\C(X)^\g$ be non-zero, and choose an affine open subvariety $\Spec A$ of $X$ contained in the open orbit of $G$ on which $f$ is regular.  Then $A$ has no non-trivial $\g$-stable ideals by \cref{affine_stable_ideals} and the remark following it.  Therefore $(f)=A$, so $f$ is non-vanishing on $A$.  However, if $x\in\Spec A(\C)$, then $f-f(x)$ is $\g$-fixed and vanishes at $x$, i.e.\ is not inveritible, so $(f-f(x))$ is a $\g$-stable ideal not equal to $A$.  Thus it must be trivial, i.e.\ $f=f(x)$, so $f$ is a constant function.
\end{proof}

As for a converse, we may state one for certain algebraic supergroups.  First we need some notation.  Suppose that $\b$ is a solvable Lie superalgebra such that $[\b_{\ol{1}},\b_{\ol{1}}]\sub[\b_{\ol{0}},\b_{\ol{0}}]$.  Then by lemma 1.37 of \cite{cheng2012dualities}, every finite-dimensional irreducible representation of $\b$ is one-dimensional.  

If $V$ is a representation of $\b$, we write $V^{(\b)}$ for the span of the $\b$-eigenvectors of $V$, which will be a semisimple representation of $\b$.  Write $\Lambda_{\b}(V)$ for the collection of characters $\lambda$ of $\b$ such that there is a $\b$-eigenvector of weight $\lambda$ in $V$.  Finally, if $\b$ acts by vector fields on the functions of a supervariety $X$, set
\[
\Lambda_{\b}^+(X):=\Lambda_{\b}(\C[X]), \ \ \ \ \ \Lambda_{\b}(X):=\Lambda_{\b}(\C(X)).
\]
Observe that if $A$ is a superalgebra which $\b$ acts on by derivations, then $A^{(\b)}$ is a subalgebra of $A$.
\begin{prop}\label{ratl_funct_open_orbit}
	Let $B$ be a solvable connected algebraic supergroup such that $[\b_{\ol{1}},\b_{\ol{1}}]\sub[\b_{\ol{0}},\b_{\ol{0}}]$ where $\b=\operatorname{Lie}(B)$.  Let $X$ be a $B$-supervariety.  If $\C(X)^{(\b)}$ is a multiplicity-free $\b$-representation such that every non-zero $f\in\C(X)^{(\b)}$ is non-nilpotent, then $X$ has an open $B$-orbit.  Equivalently, $\C(X)^{\b}=\C$ and $\C(X)^{(\b)}$ is an integral domain.
\end{prop}
\begin{proof}
	Write $\Lambda$ for the character lattice of $B$, a finitely generated free abelian group.  By our assumptions, $\C(X)^{(\b)}$ is isomorphic to the group algebra of a subgroup $\Lambda(X)$ of $\Lambda$, and hence $\Lambda(X)$ is free of some rank, say $n\in\N$.  Choose rational $B$-eigenfunctions $f_1,\dots,f_n\in \C(X)^{(\b)}$ that such that their weights form a $\Z$-basis of $\Lambda(X)$.  Then by removing the divisors of zeroes and poles of $f_1,\dots,f_n$, there exists a $B$-stable open subvariety $U$ of $X$ where $f_1,\dots,f_n$ are regular and non-vanishing, and hence $\C(X)^{(\b)}\sub\C[U]$. We may shrink $U$ further to assume that $U_0$ is normal, and we still may assume $U$ is $B$-stable.  Now apply Sumihiro's theorem (see for instance \cite{knop1989local}) using normality of $U_0$ and \cref{stability} to find a $B$-stable affine open subvariety $U'$ of $U$.  
	
	Now we claim that $U'$ is a homogeneous $B$-supervariety.  Indeed, if $I\sub \C[U']$ is a nontrivial $B$-stable ideal, then it admits a $B$-eigenfunction $f\in I$. Then $f\in\C(X)^{(\b)}$, so by assumption $f$ is invertible on $U$, and $\C[U']=(f)=I$.  We conclude by \cref{affine_stable_ideals}.
\end{proof}

\section{Quasireductive Supergroups and Hyperborels}

\begin{definition}
	A supergroup $G$ is quasireductive if $G_{0}$ is reductive.  We say a Lie superalgebra $\g$ is quasireductive if it is the Lie superalgebra of a quasireductive supergroup.
\end{definition}

\begin{example}
Many important and heavily studied supergroups are quasireductive, including $GL(m|n)$, $OSP(m|2n)$, $P(n)$, and $Q(n)$.
\end{example}
Despite having an even reductive part, the representation theory of a quasireductive group is almost never semisimple---in fact the only quasireductive supergroup with semisimple finite-dimensional representation theory, which is not a group, is $OSP(1|2n)$.  See \cite{serganova2011quasireductive} for generalities on quasireductive supergroups and their representation theory.

\subsection{Hyperborels}  In order to discuss the notion of a spherical supervariety, it is necessary that we have a well-purposed generalization of Borel subgroup (subalgebra) to the super case.  There are different notions of Borel subalgebras used for quasireductive Lie superalgebras, although the most common one seems to coincide with the definition in section 9.3 of \cite{serganova2011quasireductive}.  We use a different notion that is closer to the definition given in the beginning of chapter 3 of \cite{musson2012lie}, and agrees with this definition when the Cartan subalgebra of $\g$ is purely even.  In order to prevent confusion, we choose to call our subalgebras hyperborels.  
\begin{definition}\label{defn_borel}
	Let $\g$ be quasireductive.  A hyperborel subalgebra of $\g$ is a subalgebra $\b\sub\g$ such that
	\begin{itemize}
		\item $\b_{\ol{0}}$ is a Borel of $\g_{\ol{0}}$ in the usual sense;
		\item $[\b_{\ol{1}},\b_{\ol{1}}]\sub[\b_{\ol{0}},\b_{\ol{0}}]$; and,
		\item $\b$ is maximal with this property.
	\end{itemize}  
\end{definition}
We now give a brief discussion of this definition.
\begin{remark}\label{borel_rmk}\begin{itemize}
		\item Given a hyperborel subalgebra $\b$ and a choice of Cartan subalgebra $\h_{\ol{0}}\sub\b_{\ol{0}}$, we have $\b=\h_{\ol{0}}\ltimes\n$ where $\n$ is a nilpotent ideal.  We call $\n$ the unipotent radical of $\b$.  
		\item We may always conjugate a hyperborel $\b$ by an inner automorphism of $\g$ so that $\b_{\ol{0}}$ is a chosen Borel of $\g_{\ol{0}}$.
	\end{itemize}
\end{remark}

\begin{remark}\label{hyperborel_1_dim}
	By definition, hyperborels are solvable and all irreducible representations of them are one-dimensional (see lemma 1.37 of \cite{cheng2012dualities}).  This property is the primary way in which the notion of hyperborel subalgebra is a generalization of Borel subalgebra for reductive Lie algebras.  Further, it is this property that is of importance for us in the characterization of spherical supervarieties (\cref{spher_char} and \cref{spher_affine_char}).  
\end{remark}

Recall that for Lie superalgebras there is a notion of Cartan subalgebras (see \cite{scheunert1987invariant} and \cite{penkov1994generic}).  For a quasireductive Lie superalgebra, a Cartan subalgebra $\h$ is given by a Cartan subalgebra $\h_{\ol{0}}\sub\g_{\ol{0}}$, and then $\h$ is the centralizer of $\h_{\ol{0}}$ in $\g$.

\begin{definition}
	We say a quasireductive Lie superalgebra $\g$ is Cartan-even if for a Cartan subalgebra $\h\sub\g$, $\h=\h_{\ol{0}}$.  We say a quasireductive supergroup $G$ is Cartan-even if $\operatorname{Lie}(G)$ is.
\end{definition}

The notion of hyperborel is most natural for supergroups and superalgebras which are Cartan-even. If $\g$ is Cartan-even then the notion of hyperborel agrees with the definition of Borel given in \cite{musson2012lie}.  Further, the notion of hyperborel and Borel (as defined in \cite{serganova2011quasireductive}) coincide if $\g$ is one of the following Cartan-even superalgebras: $\g\l(m|n)$, $\s\l(m|n)$ for $m\neq n$ and $(m,n)\neq(1,1)$, $\p\s\l(n|n)$ or $\s\l(n|n)$ for $n\geq 3$, $\p(n)$, $\o\s\p(m|2n)$, or is one of the exceptional basic simple Lie superalgebras.  This is proven in Proposition 4.6.1 of \cite{musson2012lie}.  The case of $\p(n)$ is not considered there, but one can show the notions agree for this superalgebra as well (although they do not agree for the derived subalgebra of $\p(n)$). 

\begin{remark}
	If $\g$ is Cartan-even and $\b$ is a Borel subalgebra of $\g$ (as defined in \cite{serganova2011quasireductive}), then $\b$ is contained in a hyperborel subalgebra.  Indeed, a Borel subalgebra satisfies all the conditions of being a hyperborel but possibly maximality.  
\end{remark}

However if $\g$ is not Cartan-even, for instance $\g$ is the queer Lie superalgebra $\q(n)$, then hyperborels greatly differ from Borels, as they do not contain a Cartan subalgebra.

\begin{remark}
	If $\g$ is quasireductive and $\b$ a hyperborel of $\g$, then for a finite dimensional irreducible representation $V$ of $\g$, $\dim V^{(\b)}\geq 1$ by \cref{hyperborel_1_dim}.  However, it is possible that $\dim V^{(\b)}>1$, and thus we no longer have a bijective correspondence between certain characters of the Borel and finite dimensional irreducible representations. 
	
	Indeed even when $\g$ is Cartan-even this phenomenon can occur; in \cite{serganova2018representations}, a nontrivial central extension of the derived subalgebra of $\p(4)$ is considered, along with an irreducible representation $V_t$ deforming the standard representation of $\p(4)$.  If $t\neq0$, is shown that $\Lambda^2V_t$ is irreducible.  However there is a hyperborel subalgebra given by (in the notation of the paper) $\b=\g_{-2}\oplus\b_{0}\oplus\g_{1}$, where $\b_{0}$ is a Borel subalgebra of $\g_{0}$.  One can check that $\Lambda^2V_t^{(\b)}$ is two-dimensional for any $t$.  
	
	However, if a hyperborel subalgebra $\b$ contains a Borel subalgebra 	then $\dim V^{(\b)}=1$ for an irreducible representation $V$ of $\g$, by highest weight theory.  
\end{remark}

\begin{definition}\label{defn_borel_group}
	If $G$ is quasireductive, we call a subgroup $B$ a hyperborel subgroup if it is connected and gotten by integrating a hyperborel subalgebra $\b$ of $\g$.  If $\n$ is the unipotent radical, we write $N$ for the connected subgroup of $B$ it integrates to in $G$ and call it the maximal unipotent subgroup of $B$.  Finally, we write $T$ for the connected subgroup of $G$ that a chosen Cartan subalgebra $\h_{\ol{0}}\sub\b$ integrates to, which will be a maximal torus of $G_{0}$; we call $T$ a maximal torus of $B$.
\end{definition}

\begin{definition}
	If $X$ is a $G$-variety and $B$ a hyperborel of $G$ we set $\Lambda_B^+(X):=\Lambda_{\b}^+(X)$ and $\Lambda_B(X):=\Lambda_{\b}(X)$ (or simply $\Lambda^+(X)$, resp. $\Lambda(X)$ when there is no confusion).
\end{definition}

There is a canonical identification of weights of $T$ with characters of $B$ via the composition of maps $T\to B\to B/N$.  The algebra $\C[X]^N$ has a natural $T$-action, and $\Lambda_B^+(X)$ are the weights of this action under this identification.  Also observe that neither $\Lambda_B^+(X)$ nor $\Lambda_B(X)$ are a monoid or group in general, due to the presence of nilpotent functions.  For example, consider the action of an even torus on the functions of a purely odd representation of it.

\subsection{$G_0$-equivariant gradings of supervarieties}  
\begin{definition}
	Let $G_0$ be an algebraic group.  For a $G_0$-supervariety $X$, we say it has a $G_0$-equivariant grading if there exists a $G_0$-equivariant sheaf $\MM$ on $X_0$ and a $G_0$-equivariant isomorphism $X\cong\Lambda^\bullet\MM$.
\end{definition}

The following question was considered by Rothstein in \cite{rothstein1993equivariant} in the analytic setting.  Adapting the proof ideas there to the algebraic setting one can prove the following proposition.  A proof is given in the author's PhD thesis.

\begin{prop}\label{equiv_gradings}
	Let $G_0$ be a reductive group, $X$ a $G_0$-supervariety.  If $X$ is graded, then $X$ admits a $G_0$-equivariant grading.  In particular, if $G$ is quasireductive and $X$ is a graded $G$-supervariety, then $X$ admits a $G_0$-equivariant grading.
\end{prop}

\section{Spherical Supervarieties}

Let $G$  be quasireductive.
\begin{definition}
	We say a $G$-supervariety $X$ is spherical if there exists a hyperborel $B$ of $G$ with an open orbit on $X$.  If a hyperborel $B$ has an open orbit on $X$, we say that $X$ is $B$-spherical.
\end{definition}
\begin{remark}
	\begin{itemize}
		\item If a $G$-supervariety $X$ is spherical, then the $G_0$ variety $X_0$ is also spherical.
		\item Note that a spherical supervariety need not be spherical with respect to every hyperborel; in fact if $\g$ is basic classical this occurrence would be a degeneracy.
	\end{itemize}
\end{remark}

\begin{prop}\label{spher_char}
	Let $G$ be quasireductive, $B$ a hyperborel of $G$, and $X$ a supervariety.  Then $X$ is $B$-spherical if and only if $\C(X)^{(\b)}$ is a multiplicity-free $\b$-module whose nonzero elements are non-nilpotent.
\end{prop}
\begin{proof}
	This follows immediately from \cref{ratl_funct_open_orbit}.
\end{proof}
\subsection{Affine Spherical Supervarieties} In the classical case we have a characterization of affine spherical varieties by the fact the $\C[X]$ is a multiplicity-free representation.  One might hope that this generalizes to the super case.  Of course there is a first issue that for supergroups completely reducibility is a rare phenomenon to begin with.  But one might hope that perhaps $\C[X]^N$ being multiplicity-free as a $T$-module is sufficient.  This turns out to not be the case as the next examples demonstrate.

\begin{example}
	\begin{itemize}
		\item Consider the action of $GL(0|n)$ on $\C^{0|n}$ by the standard representation.  The algebra of functions is $\Lambda^\bullet(\C^n)^*$, which is completely reducible and multiplicity-free. However, there is only one point and the orbit of it under the whole group is itself, so this space is not spherical.
		\item An example which has a nontrivial even part is given by considering $G=OSP(1|2)$ and letting $X=OSP(1|2)/T$, where $T$ is a maximal torus of $G_{0}$.  By the representation theory of $OSP(1|2)$ and Frobenius reciprocity, $\C[X]\cong\bigoplus\limits_{n\geq 0}\Pi^nL(n)$, where $L(n)$ is the irreducible representation of highest weight $n$ with even highest weight vector.  Hence $\C[X]$ is completely reducible and multiplicity-free.  However, no hyperborel admits an open orbit since the odd dimension of $X$ is 2 while the odd dimension of any hyperborel is 1.
	\end{itemize} 
\end{example}
The next theorem demonstrates that the issue with the above two spaces is that some of the highest weight functions are nilpotent.
\begin{thm}\label{spher_affine_char}
	Let $X$ be an affine $G$-supervariety, $B$ a hyperborel of $G$ with maximal unipotent subgroup $N$ and maximal torus $T$.  Then the following are equivalent:
	\begin{enumerate}
		\item $X$ is spherical for $B$.
		\item $X_{0}$ is spherical for $B_{0}$, and every nonzero $B$-highest weight function in $\C[X]$ is non-nilpotent.
		\item Every nonzero $B$-highest weight function in $\C[X]$ is non-nilpotent, and $\dim\C[X]^{N}_{\lambda}\leq 1$ for all weights $\lambda$ of $T$.
		\item $\C[X]^N$ is an even commutative algebra without nilpotents, and the natural $T$-action is multiplicity-free.
	\end{enumerate}
\end{thm}

\begin{proof}
	(1)$\implies$(2): Let $x\in X(\C)$ be such that $a_x:B\to X$ is a submersion, so that $a_x^*$ is injective.  In $\C[B]$, all $B$-highest weight functions are non-nilpotent, and therefore the same must be true of the functions on $X$.
	
	(2)$\implies$(3): Since $X_{0}$ is spherical for $B_{0}$, we have $\dim\C[X_{0}]^{(\b_{\ol{0}})}_{\lambda}\leq 1$ for all $\lambda$.  Since the $B$-highest weight functions are non-nilpotent, the restriction map $\C[X]^{(\b)}\to\C[X_0]^{(\b_{\ol{0}})}$ is injective, and we are done.
	
	(3)$\implies$(4): We see that $\C[X]^N$ is the subalgebra generated by the $B$-highest weight functions, so this is clear.
	
	(4)$\implies$(1): Let $S$ be the submonoid of the character lattice of $T$ determined by $\C[X]^N$.  Then because group generated by $S$ is finitely generated, of rank say $m$, there exists weights $\lambda_1,\dots,\lambda_m\in S$ such that the monoid generated by $S$ and $-\lambda_1,\dots,-\lambda_m$ is a group.  Then if we let $U$ be the non-vanishing locus of $f_{\lambda_1},\dots,f_{\lambda_m}$, all $B$-eigenfunctions in $\C[U]$ will be invertible.  Further, this open subscheme $U$ will be $B$-stable.  Choose a point $x\in U(\C)$, and consider the orbit map $a_x:B\to X$.  Since all $f_{\lambda}$ become units on $U$, they must not be in the kernel of $a_x^*$.  But if $a_x^*$ is not injective, the kernel will contain a $B$-highest weight function, a contradiction.  Therefore $a_x$ must be a submersion, and so $X$ is spherical.  
\end{proof}

\begin{definition}
	If a $G$-supervariety $X$ is $B$-spherical, define the rank of $X$ to be the rank of the lattice $\Lambda_B(X)$.
\end{definition}
A corollary of the proof of the above proposition is the following.  
\begin{cor}
	If $X$ is $B$-spherical of rank $m$, there exists $m$ $B$-highest weight functions $f_{\lambda_1},\dots,f_{\lambda_m}\in\C[X]$ such that their common non-vanishing set is the open $B$-orbit.
\end{cor}

\begin{cor}\label{socle_mult_free}
	If $X$ is spherical, the socle of $\C[X]$ is multiplicity free.
\end{cor}

\begin{proof}
	Suppose that an irreducible representation $L$ shows up with multiplicity greater than $1$.  If $B$ is a hyperborel for which $X$ is $B$-spherical, there will be two $B$-eigenfunctions of the same weight in $\C[X]$.  This contradicts (3) of \cref{spher_affine_char}.
\end{proof}

Now suppose that $X$ is an affine $B$-spherical supervariety and $U$ is the open $B$-orbit.  By the reasoning given in the proof of \cref{ratl_funct_open_orbit}, we know that all rational $\b$-eigenfunctions will be regular (and in fact non-vanishing) on $U$.  Hence $\C[U]^N=\C(X)^{(\b)}$, and because these functions are all non-nilpotent we have 
\[
\C(X)^{(\b)}=\C[U]^N\cong \C[U_{0}]^{N_0}=\C(X_0)^{\b_{\ol{0}}}
\] 
by restriction of functions.  Further, these algebras are all isomorphic to group algebras on $\Lambda_B(X)$, a finitely generated free abelian subgroup of the character lattice of $T$.

Now on all of $X$, restriction induces an injective map $\C[X]^N\to \C[X_0]^{N_0}$, and hence an inclusion $\Lambda^+_B(X)\sub\Lambda^+_{B_0}(X_0)$, and thus 	$\Lambda^+_{B}(X)$ will be a submonoid of $\Lambda^+_{B_0}(X_0)$.  Note that $\C[X]^N$ is the monoid algebra on $\Lambda_{B}^+(X)$ and $\C[X_{0}]^{N_0}$ is the monoid algebra on $\Lambda_{B_0}^+(X_0)$.

It is a classical fact about spherical varieties that $\C[X_{0}]^{N_{0}}$  is finitely generated, so choose generators $g_1,\dots,g_n$ which are $\b_{\ol{0}}$-eigenfunctions.  Note that $U_0$ is precisely the non-vanishing locus of these functions.  We may uniquely lift these to $\b$-eigenfunctions $f_{1},\dots,f_{n}$ on $U$.  

Let us now assume in addition that $X$ is graded. By \cref{equiv_gradings} we may choose a $G_0$-equivariant grading of $X$.  Thus we may write $\C[X]=\Lambda^\bullet M$, where $M$ is a finitely generated $G_0$-equivariant $\C[X_0]$-module.  Let us assume the largest non-zero exterior power of $M$ is $q$.  Then we may write 
\[
f_{i}=g_{i}+m_{i1}+\dots+m_{iq}\text{ where }m_{ij}\in\Lambda^jM_{g_1\cdots g_n}.
\]
Here $M_{g_1\cdots g_n}$ is the localization of $M$ to the non-vanishing locus of $g_1,\dots,g_n$.  We may do this because since $f_i$ is a $\b$-eigenvector, each $m_{ij}$ must be a $\b_{\ol{0}}$-eigenvector and it must be regular on the open $B_0$-orbit.  Now the obstruction to regularity of $f_i$ is the poles of $m_{ij}$ along $g_1\cdots g_n=0$.  For each $m_{ij}$, there exists $q_{ij}\in\N$ such that $(g_1\cdots g_k)^{q_{ij}}m_{ij}\in\Lambda^jM$.  By choosing an integer $p$ larger than $q+\max\limits_{i,j}q_{ij}$, we now have:

\begin{prop}\label{weight_lattice}
	If $X$ is graded then there exists an integer $p>0$ such that $f_1^{r_1}\cdots f_n^{r_n}$ is regular whenever $r_1,\dots,r_n\geq p$.
\end{prop}

\begin{proof}
	Expanding out the product, one sees that for any integer $p$ chosen as described in the paragraph before the proposition, the poles will be resolved.
\end{proof}

\begin{cor}\label{weight_lattice_dense}
	If $X$ is graded then the set $\Lambda^+_B(X)$, which is a submonoid of $\Lambda^+_{B_0}(X_0)$, generates $\Lambda_B(X)=\Lambda_{B_0}(X_0)$ as a group.  Further it is Zariski dense in the vector space spanned by its weights.
\end{cor}

\begin{proof}
	By \cref{weight_lattice}, $\Lambda^+_B(X)$ contains the lattice points of a translated orthant of $\R\otimes_\Z\Lambda_B(X)$, and so the results follow.
\end{proof}

Write $X//N:=\Spec \C[X]^N$.  Then by (4) of \cref{spher_affine_char}, $X//N$ is an even variety and admits a natural $T$-action such that $\C[X//N]$ is a multiplicity-free $T$-module.  In particular, $X//N$ has an open $T$-orbit, hence is essentially a toric variety but that it need not be normal or Noetherian.  Indeed, we observe it is isomorphic to the group algebra of $\Lambda^+_B(X)$, so being normal is equivalent to this monoid being saturated, and being Noetherian is equivalent to the monoid being finitely generated.  We now present examples showing how these properties can fail.

\begin{example}
	Consider the action of $G=GL(1|2)$ on $X=S^2\C^{1|2}$ as the second symmetric power of the standard representation.  This is a spherical supervariety as one can check (this was checked in \cite{sherman2020spherical}), and is spherical exactly with respect to the hyperborels $B^+$ and $B^-$ of upper and lower triangular matrices, respectively.  The coordinate ring $\C[X]$ is a supersymmetric polynomial algebra given by $S^\bullet(S^2(\C^{1|2})^*)$ as both an algebra and a $G$-module.  
	
	As a $G_0=GL(1)\times GL(2)$-representation $X_0$ is a sum of two one-dimensional representations of distinct weights.  Therefore the $B_0$-highest weight functions of $X_0$ are the monomials in two $G_0$-eigenfunctions $x,y$, where we let $x$ have weight $\lambda$ and $y$ have weight $\mu$. Let $\xi,\eta\in(S^2\C^{1|2})^*_{\ol{1}}$ be odd weight vectors of weights $\alpha,\beta$.  Then $\C[X]=\C[x,y,\xi,\eta]$.  One can show that $\xi\eta$ is a $G_0$-eigenvector of weight $\lambda+\mu$, and so one can show that for any hyperborel $B$ the rational $B$-eigenfunctions on $X$ are, up to scalar, all of the form:
	\[
	f_{ij}=x^iy^j+c_{ij}x^iy^j\frac{\xi\eta}{xy}
	\]
	where $i,j\in\Z$ and $c_{ij}\in\C$ is a coefficient in $\C$ to be determined depending on the choice of hyperborel.  For the hyperborel $B^+$, we find that $c_{ij}=i$ and for $B^-$ we find that $c_{ij}=-j$.  These values for $c_{ij}$ tell us which rational $B$-eigenfunctions are regular on all of $X$, or equivalently tell us what $\Lambda^{+}_{B^{\pm}}(X)$ are.  We draw the two monoids below to visualize the result:
	
	
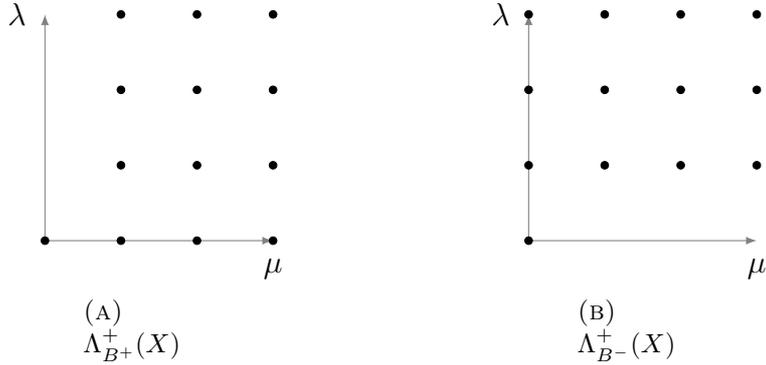
\begin{figure}[!ht]
		\subfloat[$\Lambda_{B^+}^+(X)$]{
			\begin{tikzpicture}
			\coordinate (Origin)   at (1,0);
			\coordinate (XAxisMin) at (1,0);
			\coordinate (XAxisMax) at (4,0);
			\coordinate (YAxisMin) at (1,0);
			\coordinate (YAxisMax) at (1,3);
			\draw [thin, gray,-latex] (XAxisMin) -- (XAxisMax);
			\draw [thin, gray,-latex] (YAxisMin) -- (YAxisMax);
			
			\foreach \x in {1,2,3}{
				\foreach \y in {0,1,2,3}{
					\node[draw,circle,inner sep=1pt,fill] at (\x+1,\y) {};
				}
			}
			\node[below=0.1cm] at (XAxisMax) {$\mu$};
			\node[left=0.1cm] at (YAxisMax) {$\lambda$};
			
			\node[draw,circle,inner sep=1pt,fill] at (1,0) {};
			\end{tikzpicture}}
		\qquad
		\qquad
		\qquad
		\subfloat[$\Lambda_{B^-}^+(X)$\label{subfig-2:dummy}]{%
			\begin{tikzpicture}
			\coordinate (Origin)   at (0,0);
			\coordinate (XAxisMin) at (0,0);
			\coordinate (XAxisMax) at (3,0);
			\coordinate (YAxisMin) at (0,0);
			\coordinate (YAxisMax) at (0,3);
			\draw [thin, gray,-latex] (XAxisMin) -- (XAxisMax);
			\draw [thin, gray,-latex] (YAxisMin) -- (YAxisMax);
			
			\foreach \x in {0,1,2,3}{
				\foreach \y in {1,2,3}{
					\node[draw,circle,inner sep=1pt,fill] at (\x,\y) {};
				}
			}
			\node[draw,circle,inner sep=1pt,fill] at (0,0) {};
			
			\node[below=0.1cm] at (XAxisMax) {$\mu$};
			\node[left=0.1cm] at (YAxisMax) {$\lambda$};
			\end{tikzpicture}
		}
	\end{figure}
	For comparison, the monoid $\Lambda^+_{B_0}(X)$ for any Borel subgroup $B_0$ of $G$ consists of all the lattice points that are a nonnegative linear combination of $\lambda$ and $\mu$.  This example demonstrates that $\Lambda_B^+(X)$ need not be finitely generated as neither of the above monoids are finitely generated. 
\end{example}

\begin{example}
	Consider the action of $G=OSP(3|4)\times OSP(3|4)$ on $X=OSP(3|4)$ by left and right multiplication.  The notion of Borel and hyperborel coincide for both $OSP(3|4)$ and $OSP(3|4)\times OSP(3|4)$.  If $B$ is a Borel of $OSP(3|4)$, then $X$ is $B\times B^{-}$-spherical where $B^{-}$ is the opposite Borel of $B$.  As we will see in the examples section, $\Lambda^+_{B\times B^{-}}(X)$ will be exactly the $B$-dominant weights of $OSP(3|4)$.  Now if we choose the Borel determined by the simple roots $\delta_1-\delta_2,\delta_2-\epsilon_1,\epsilon_1$ as described in section 1.3.3 of \cite{cheng2012dualities}, then by theorem 2.11 of \cite{cheng2012dualities} the weight $\lambda=\epsilon_1+\epsilon_2+\delta_1+\delta_2$ is not dominant while $k\lambda$ is dominant for $k\geq 2$.  Thus $\Lambda^+(X)$ is will not be saturated in this case.
\end{example}

\subsection{Spherical quasiprojective supervarieties}
Let $G$ be quasireductive, $X$ a quasiprojective $G$-supervariety, and $B$ a hyperborel subgroup.  
\begin{thm}\label{qproj_spher_char}
	Let $X$ be a quasiprojective $G$-supervariety, $B$ a hyperborel subgroup of $G$ with unipotent radical $N$ and maximal torus $T$.  Assume that there exists a very ample $B$-equivariant line bundle on $X$.    Then the following are equivalent:
	\begin{enumerate}
		\item $X$ is spherical for $B$.
		\item $X_{0}$ is spherical for $B_{0}$, and for every $B$-equivariant line bundle $\LL$, the non-zero elements of $\Gamma(X,\LL)^{N}$ are non-vanishing at some point.
		\item For every $B$-equivariant line bundle $\LL$, the non-zero elements of $\Gamma(X,\LL)^{N}$ are non-vanishing at some point, and $\dim\Gamma(X,\LL)^{N}_{\lambda}\leq 1$ for all weights $\lambda$ of $T$.
	\end{enumerate}
\end{thm}
\begin{proof}
	(1)$\implies$(2) Let $\LL$ be a $B$-equivariant line bundle on $X$.  Write $U$ for the open $B$-orbit on $X$.  Then if $\sigma\in\Gamma(X,\LL)$ is a $B$-eigenvector, restricting to $U$ it spans a nonzero $B$-submodule of $\Gamma(U,\LL)$.  Since $U$ is a homogeneous $B$-supervariety, by \cref{homog_line_bundle_char} $\sigma$ must generate $\LL|_{U}$, and thus must be non-vanishing on $U$.
	
	(2)$\implies$(3) It suffices to check that for a $B$-equivariant line bundle $\LL$, $\dim\Gamma(X,\LL)^{N}_{\lambda}\leq 1$ for all weights $\lambda$ of $T$. Since $X_0$ is spherical for $B_0$ and $i:X_0\to X$ is $G_0$-equivariant, $i^*\LL$ is a $B_0$-equivariant line bundle, and the pullback morphism $\Gamma(X,\LL)\to\Gamma(X_0,i^*\LL)$ is $B_0$-equivariant.  The map $\Gamma(X,\LL)^N$ $\to\Gamma(X_0,i^*\LL)^{N_0}$ is injective by assumption, and since $X_0$ is spherical we have $\dim\Gamma(X_0,i^*\LL)^{N_0}_{\lambda}\leq 1$ for all weights $\lambda$. 
	
	(3)$\implies$(1) Let $\LL$ be a $B$-equivariant very ample line bundle on $X$.  Then if $X$ does not have an open $B$-orbit, it must follow that $\C(X)^{(\b)}$ either has a nilpotent function or is not multiplicity-free.  If $f\in \C(X)^{(\b)}$ is homogeneous, there exists $n>0$ and a homogeneous global section $s\in\Gamma(X,\LL^{\otimes n})$ such that $fs\in\Gamma(X,\LL^{\otimes n})$ is also a global section.  Let $V\sub\Gamma(X,\LL^{\otimes n})$ be the subspace of sections $s$ such that $fs\in\Gamma(X,\LL^{\otimes n})$.  Then $V$ is a $B$-submodule of $\Gamma(X,\LL^{\otimes n})$ and thus admits a non-zero $B$-eigenvector $s_2$.  Let $s_1:=fs_2$.  Then by construction $s_1$ is also a $B$-eigenvector.  In particular, $s_1$ and $s_2$ both are non-vanishing at some point by assumption, and thus $f=s_1/s_2$ is also non-vanishing at some point and therefore cannot be nilpotent.  
	
	If $f,g\in\C(X)^{(\b)}$ are $\b$-eigenvectors with the same weight for the action of $\b$, then $f/g$ will be $\b$-invariant.  Thus by our construction there exists an $n>0$ and $s_1,s_2\in \Gamma(X,\LL^{\otimes n})^{(\b)}$ such that $f=s_1/s_2$, and thus $s_1$ and $s_2$ have the same weight for $\b$.  But by assumption $\Gamma(X,\LL^{\otimes n})^N$ is a multiplicity-free $T$-module, so we obtain a contradiction.  This completes the proof.
\end{proof}

\section{Examples}

We present some examples of spherical supervarieties.

\subsection{Spherical Representations}  Irreducible spherical representations of reductive algebraic groups were originally classified by Kac in \cite{kac1980some}.  In \cite{sherman2020spherical}, the author classified all indecomposable spherical representations of the groups $GL(m|n)$, $OSP(m|2n)$, $P_{n|n}$, and the basic exceptional simple groups.  The case of $Q(n)$ is also looked at, however a different notion of spherical was used for this supergroup there.  

We found there are a few infinite families of irreducible representations, along with certain small exceptional cases.  Below is a table of the infinite families; for the rest, we refer the reader to the paper.  We write $GL_{m|n}$, $OSP_{m|2n}$, and $P_{n|n}$ respectively for the standard representations of $GL(m|n)$, $OSP(m|2n)$, and $P(n)$ respectively.  We also state the dimension of the representation and whether the algebra of functions on it is completely reducible.

\renewcommand{\arraystretch}{1.25}

\[	\begin{tabular}{ |c|c|c| } 
\hline 
$V$ & $\text{dim}^sV$ & Completely red.? \\
\hline
$GL_{m|n}$ & $(m|n)$ & Yes \\
\hline 
$S^2GL_{m|n}$ & $(\frac{n(n-1)}{2}+\frac{m(m+1)}{2}|mn)$& Yes\\ 
\hline
$\Pi S^2GL_{n|n}$ &  $(n^2|n^2)$& Yes\\
\hline
$\Pi S^2GL_{n|n+1}$& $(n(n+1)|n(n+1))$& Yes \\
\hline
$OSP_{m|2n}$, $m\geq 2$ & $(m|2n)$ &\makecell{Iff $m$ is odd \\ or $m>2n$} \\
\hline
$\Pi OSP_{m|2n}$ & $(2n|m)$ & Yes \\
\hline 
$\Pi P_{n|n}$ & $(n|n)$ & No  \\
\hline
\end{tabular}
\]
\renewcommand{\arraystretch}{1}

\subsection{Symmetric Supervarieties}  Let $\g$ be quasireductive.  Given an involution $\theta$ of $\g$, we write $\k=\g^\theta$ for the fixed points of $\theta$, and call the pair $(\g,\k)$ a supersymmetric pair.  If $G$ is a Lie supergroup and $K$ a subgroup with $\operatorname{Lie}(G)=\g$ and $\operatorname{Lie}(K)=\k$, we call the coset space $G/K$ a symmetric supervariety.  

In the classical world, symmetric varieties for reductive groups are always spherical by the Iwasawa decomposition.  We recall how this decomposition works now, generalizing it to the super case.  We keep the same notation, letting $\g$ be quasireductive, $\theta$ an involution of $\g$ with fixed points $\k$ and $(-1)$-eigenspace $\p$.  Then let $\a\sub\p$ be a maximal toral subalgebra of $\p$, i.e.\ a maximal abelian subspace of $\p_{\ol{0}}$ with the property that the elements of $\a$ are semisimple in $\g_{\ol{0}}$.  Then we may decompose $\g$ into weight spaces under the adjoint action of $\a$.  Write $\Sigma\sub\a^*$ for the set of non-zero weights under this action.  Choosing a generic hyperplane we obtain a subset $\Sigma^+\sub\Sigma$ of positive weights, and we define
\[	
\n=\bigoplus\limits_{\alpha\in\Sigma^+}\g_{\alpha}.
\]
Write $C(\a)$ for the centralizer of $\a$ in $\g$.  Then we have $C(\a)=C(\a)\cap\k\oplus C(\a)\cap\p$.  
\begin{prop}
	The condition $C(\a)\cap\p=\a$ is equivalent to the following decomposition of $\g$:
	\[
	\g=\k\oplus\a\oplus\n.
	\]
	We call such a decomposition an Iwasawa decomposition of the symmetric pair $(\g,\k)$ (or of the involution $\theta$).
\end{prop}
It is a well-known theorem that if $\g=\g_{\ol{0}}$ is reductive then every symmetric pair has an Iwasawa decomposition (see for instance section 26.4 of \cite{timashev2011homogeneous}).  However in the super world this no longer remains true.  In particular, it is possible for $C(\a)\cap\p_{\ol{1}}\neq0$.  However, we do have the following:
\begin{thm}
	If $\g$ is a basic classical simple Lie superalgebra and $\theta$ is an involution that preserves the invariant bilinear form on $\g$, then either $\theta$ or $\delta\circ\theta$ has an Iwasawa decomposition, where $\delta\in\Aut(\g)$ is the grading automorphism $\delta(x)=(-1)^{\ol{x}}x$.  
\end{thm}  
This result is proven in \cite{alex2020iwasawa} using the framework of generalized root systems as developed by Serganova in \cite{serganova1996generalizations}.  There the restricted root systems of symmetric pairs is also studied as well as the notion of Satake diagrams.  The significance of an Iwasawa decomposition for our purposes is that
\begin{thm}
	If a symmetric pair $(\g,\k)$ admits an Iwasawa decomposition, then there exists a hyperborel $\b$ of $\g$ such that $\b+\k=\g$.  In particular, a symmetric supervariety $G/K$ constructed from this symmetric pair is spherical.
\end{thm}
\begin{proof}
	Write $\g=\k\oplus\a\oplus\n$ for the Iwasawa decomposition.  Write $\Sigma^+\sub\a^*$ for the positive weights defining $\n$.  Let $\h_{\ol{0}}\sub\g_{\ol{0}}$ be a Cartan subalgebra containing $\a$.  Write $\Delta\sub\h^*$ for the roots of $\g$ with respect to $\h$.  Then we have a natural projection map $\h^*\to\a^*$ inducing a map $\Delta\to\Sigma\cup\{0\}$.  Choose a generic hyperplane in $\Delta$ so that the image of $\Delta^+$ under this projection lands in $\Sigma^+\cup\{0\}$.  Then consider
	\[
	\b'=\h\oplus\bigoplus\limits_{\alpha\in\Delta^+}\g_{\alpha}
	\]
	Then $\b'$ satisfies all the properties of a hyperborel apart possibly from maximality, and thus is contained in a hyperborel subalgebra $\b$ of $\g$.  Further, $\a\oplus\n\sub\b'\sub\b$, and therefore $\k+\b=\g$ completing the proof.
\end{proof}
Below we list all supersymmetric pairs (up to conjugacy) for the algebras $\g\l(m|n)$, $\o\s\p(m|2n)$, $\p(n)$, and the simple basic exceptional algebras.  For all cases but $\a\b(1|3)$, $\g(1|2)$, and $\mathfrak{d}(2|1;\alpha)$ we refer to the classification in \cite{serganova1983classification}; these remaining the cases were communicated to the author by Serganova.  

For each we state whether or not the pair is spherical as well as whether it admits an Iwasawa decomposition.  Note that we only consider involutions of $\g\l(m|n)$ that fix the center.

\renewcommand{\arraystretch}{1.2}

\begin{tabular}{ |c|c|c| } 
\hline 
Symmetric Pair & Spherical? & Iwasawa decomposition? \\
\hline
$(\g,\g_{\ol{0}})$ & Iff $\g=\g_{\ol{0}}$ & Iff $\g=\g_{\ol{0}}$  \\
\hline
\makecell{$(\g\l(m|n)$,\\ $\g\l(r|s)\times\g\l(m-r|n-s))$} & \makecell{Iff $r\geq m-r$ and $s\geq n-s$\\ or $r\leq m-r$ and $s\leq n-s$}    & Same condition  \\
\hline
$(\g\l(m|n),\o\s\p(m|2n))$ & Yes & Yes  \\
\hline 
$(\g\l(n|n),\p(n))$ & Yes & No \\ 
\hline
$(\g\l(n|n),\q(n))$ & Yes & Yes\\
\hline
\makecell{$(\o\s\p(m|2n)$,\\$\o\s\p(r|2s)\times\o\s\p(m-r,2n-2s))$}& \makecell{Iff $r\geq m-r$ and $s\geq n-s$\\ or $r\leq m-r$ and $s\leq n-s$} & Same condition \\
\hline
$(\o\s\p(2m,2n),\g\l(m|n))$ & Yes & Yes \\
\hline
$(\p(n),\p(r)\times\p(n-r))$ & Iff $r=1$ & No \\
\hline 
$(\p(n),\g\l(r|n-r))$ & Iff $n=2,3$ & No  \\
\hline
	$(\mathfrak{d}(2|1;\alpha),\o\s\p(2|2)\times\s\o(2))$ &  Yes & Yes \\
\hline
$(\a\b(1|3),\g\o\s\p(2|4))$ & Yes & Yes\\
\hline
$(\a\b(1|3),\s\l(1|4))$ & Yes & Yes\\
\hline
$(\a\b(1|3),\mathfrak{d}(2|1;2))$ & Yes & Yes\\
\hline
$(\g(1|2),\mathfrak{d}(2|1;3))$ & Yes &  Yes\\
\hline
$(\g(1|2),\o\s\p(3|2)\times\s\l_2)$ & No &  No\\
\hline
\end{tabular}

\subsection{$G$ as a spherical supervariety}  Let $G$ be a quasireductive supergroup.  Then $G\times G$ acts homogeneously on $G$ by left and right translation, and this identifies $G$ as a symmetric supervariety with respect to the involution $\theta$ of $G\times G$ which swaps the factors. 

Some is already known about the structure of $\C[G]$ as a representation.  For instance, in \cite{serganova2011quasireductive}, the structure as a $G$-module under left translation was computed and was shown to be a sum of injective modules.  In \cite{la2012donkin} a filtration of $\C[GL(m|n)]$ as a $G\times G$-module was constructed following the ideas of Donkin and Koppinen in the modular case, using the highest weight category structure of representations of $GL(m|n)$.  Serganova's result on the structure of $\C[G]$ under left translation also follows from Green's work on coalgebras in \cite{green1976locally}, generalized to the setting of supercoalgebras.  We state some further results on $\C[G]$ looking at its structure as a $G\times G$-module that are straightforward extensions of results found in \cite{green1976locally}, in particular on indecomposable block summands and the socle of $\C[G]$.  Then we state a result that describes the Loewy layers of the socle filtration of $\C[G]$ (\cref{loewy_layers}) which the author has not found the literature.

\begin{thm} Let $\g$ be a quasireductive Lie superalgebra and consider the supersymmetric pair $(\g\times\g,\g)$ defined by the involution $\theta$ of $\g\times\g$ which swaps the factors.  Then this supersymmetric pair admits an Iwasawa decomposition if and only if $\g$ is Cartan-even.  
\end{thm}

\begin{proof}
	In this case a maximal toral subalgebra of the (-1)-eigenspace is given by $\a=\{(h,-h):h\in\h_{\ol{0}}\}$ where $\h_{\ol{0}}\sub\g_{\ol{0}}$ is a Cartan subalgebra of $\g_{\ol{0}}$.  Therefore the centralizer of $\a$ is just the centralizer of $\h_{\ol{0}}\times\h_{\ol{0}}$ in $\g\times\g$.  This is equal to $\h_{\ol{0}}\times\h_{\ol{0}}$ if and only if $\h_{\ol{0}}$ is a Cartan subalgebra of $\g$, i.e.\ $\g$ is Cartan-even. 
\end{proof}

\begin{prop}
	If $G$ is Cartan-even, the finite-dimensional irreducible representations of $G\times G$ are exactly those of the form $L_1\boxtimes L_2$ for finite-dimensional irreducible representations $L_1,L_2$ of $G$.
\end{prop} 
\begin{proof}
	A representation of this form is irreducible because $\End_G(V_i)\cong\C$ for each $i$ and the Jacobson density theorem.  Conversely, if $L$ is an irreducible representation of $G\times G$ then after choosing a Borel subgroup, it has a highest weight $\lambda_1+\lambda_2$, where $\lambda_i$ is a weight of $i$th copy of $G$ in the direct product.  Thus $L=L_B(\lambda_1)\boxtimes L_B(\lambda_2)$.
\end{proof}
\begin{definition}
	Let $V$ be a finite-dimensional $G$-module corresponding to the coaction $V\to\C[G]\otimes V$.  Define $\epsilon_V:V\boxtimes V^*\to\C[G]$ to be the canonical $G\times G$-equivariant map corresponding to the coaction.  Notice that it is always nonzero if $V$ is nonzero.  Equivalently, $\epsilon_V$ may be defined by Frobenius reciprocity; it is the unique element of $\Hom_{G\times G}(V\boxtimes V^*,\C[G])$ that corresponds to the natural pairing $V\otimes V^*\to\C$ under the isomorphism
	\[
	\Hom_{G\times G}(V\boxtimes V^*,\C[G])\cong\Hom_{G}(V\otimes V,\C)
	\]
\end{definition}
\begin{remark}
	If $V$ is a finite-dimensional $G$-representation then there is a canonical isomorphism of $G\times G$-modules $V\boxtimes V^*\cong(\Pi V)\boxtimes(\Pi V)^*$, and this map factors $\epsilon_V$ through $\epsilon_{\Pi V}$.  In particular, $\operatorname{Im}\epsilon_{V}=\operatorname{Im}\epsilon_{\Pi V}$.
\end{remark}

For the rest of this section we will assume that $G$ is Cartan-even.  Given an irreducible representation $L$ of $G$, the map $\epsilon_L:L\boxtimes L^*\to\C[G]$ is injective by irreducibility and the fact that $\epsilon_L$ is not the zero map.  In this way we obtain a natural inclusion 
\[
\bigoplus\limits_L L\boxtimes L^*\sub\operatorname{soc}(\C[G]),
\]
where the sum runs over all irreducible representations of $G$ up to parity. We now go about showing this is the entire socle. 

Let $B'$ be a Borel subgroup of $G$ (as defined in \cite{serganova2011quasireductive}) and $(B')^{-}$ its opposite Borel.  Let $B$ be a hyperborel subgroup containing $B'$ and $B^{-}$ a hyperborel subgroup containing $(B')^-$.  Then $B\times B^-$ is a hyperborel of $G\times G$, and $G$ is is $B\times B^-$-spherical.  Further, $(B^-)_0$ is the Borel subgroup of $B_0$ in $G_0$.  
\begin{lemma}
	If $L$ is an irreducible representation of $G$, then $L^{(B)}=L^{(B')}$.  
\end{lemma}
\begin{proof}
	Indeed $L^{(B)}\sub L^{(B')}$ but by \cref{hyperborel_1_dim} $1\leq \dim L^{(B)}\leq \dim L^{(B')}=1$.  
\end{proof}

\begin{definition} For a hyperborel subgroup $B$ of $G$, we say an integral weight $\lambda$ is $B$-dominant if there exists an irreducible representation $L$ of $G$ such that $\Lambda_B(L)=\{\lambda\}$.
\end{definition}
Recall that (for instance by the Peter-Weyl theorem), 
\[
\Lambda^+_{B_0\times(B^-)_0}(G_0)=\{(\lambda,-\lambda):\lambda\text{ is a }B_0\text{-dominant weight}\}.
\]

\begin{lemma}
	We have
	\[
	\Lambda^+_{B\times B^-}(G)=\{(\lambda,-\lambda):\lambda\text{ is a }B\text{-dominant weight}\}.
	\]
\end{lemma}
\begin{proof}
	By the inclusion $\Lambda^+_{B\times B^-}(G)\sub\Lambda^+_{B_0\times(B^-)_0}(G_0)$ we know that $\Lambda^+_{B\times B^-}(G)$ must be contained in the RHS.  However our socle computation above shows that $L(\lambda)\boxtimes L(\lambda)^*\sub \C[G]$ for all $B$-dominant weights $\lambda$, and this is exactly the $G\times G$ irreducible representation of highest weight $(\lambda,-\lambda)$.
\end{proof}  

\begin{cor}\label{socle_G}
	$\operatorname{soc}(\C[G])\cong\bigoplus\limits_L L\boxtimes L^*$, where the sum runs over all irreducible representations of $G$ up to parity.
\end{cor}

We explain further the structure of $\C[G]$. Let $\mathsf{Rep}(G)$ denote the category of finite-dimensional representations of $G$.  Then we may decompose $\mathsf{Rep}(G)$ into a sum of simple blocks, where a block $\BB$ is an abelian subcategory of $\mathsf{Rep}(G)$ such that if $\BB'$ is another block distinct from $\BB$, then $\Ext^i(V,V')=\Ext^i(V',V)=0$ for all $i$ and all objects $V$ of $\BB$ and $V'$ of $\BB'$.  A block $\BB$ is simple if it cannot be decomposed into a sum of smaller, nontrivial blocks.  Notice that every block must contain an irreducible representation.

Given a block $\BB$ of $G$, we denote by $\Pi\BB$ the block consisting of all $G$-modules $\Pi V$ where $V$ is in $\BB$. If we write $\text{Bl}_G$ for the set of blocks of $G$, we want to consider the set $\text{Bl}_G/\sim$ where $\sim$ is the equivalence relation on blocks generated by $\BB\sim\Pi\BB$ for all blocks $\BB$.  For $\BB\in\text{Bl}_G/\sim$, we write $\operatorname{Irr}\BB$ for the set of irreducible representations that appear in $\BB$ up to parity.  The following is an analogue of theorem (1.5g) part (ii) and theorem (1.6a) in \cite{green1976locally}.
\begin{prop}
	We have as a $G\times G$-module 
	\[
	\C[G]=\bigoplus\limits_{\BB\in\text{Bl}_G/\sim}M_{\BB}
	\]
	where $M_{\BB}$ is an indecomposable $G\times G$-module given by
	\[
	M_{\BB}=\sum\limits_{V\in\BB}\operatorname{Im}\epsilon_V.
	\] 
	Further,
	\[
	\operatorname{soc}(M_{\BB})=\bigoplus\limits_{L\in\operatorname{Irr}\BB}L\boxtimes L^*.
	\]
\end{prop}

\begin{remark} It follows $M_{\BB}$ is finite-dimensional if and only if $\operatorname{Irr}\BB$ is finite.  This example shows another phenomenon that may occur in the super case: given a spherical $G$-supervariety $X$, $\C[X]$ need not be a direct sum of finite-dimensional $G$-modules.  
\end{remark}

We can say more about the socle filtration of $M_{\BB}$, and thus of $\C[G]$.  Recall that for a finite-dimensional $G$-module $V$, the Loewy length of $V$, which we write as $\ell\ell(V)$, is defined to be the length of a minimal semisimple filtration of $V$ (or equivalently the length of the socle or radical filtration of $V$).  The first of the following results is an analogue of what was essentially known in \cite{green1976locally} for coalgebras.  The author has not found the second result in the literature.  A proof is given for more general coalgebras in \cite{sherman2020layers}.
\begin{thm}\label{loewy_layers}For each block $\BB\in\operatorname{Bl}_G/\sim$ we have:
	\begin{itemize}
		\item \[
		\operatorname{soc}^kM_{\BB}=\sum\limits_{V\in\BB,\ \ell\ell(V)\leq k}\operatorname{Im}\epsilon_V
		\]
		\item For simple $G$-modules $L,L'$ which lie in a block of the equivalence class $\BB$, we have
		\begin{eqnarray*}
			[\operatorname{soc}^kM_{\BB}/\operatorname{soc}^{k-1}M_{\BB}:L'\boxtimes L^*]& = &[L':\operatorname{soc}^kI(L)/\operatorname{soc}^{k-1}I(L)]\\
																					 	 &\hspace{-12em} = &\hspace{-6em}\dim\Hom_G(\PP(L'),\operatorname{soc}^kI(L)/\operatorname{soc}^{k-1}I(L))
		\end{eqnarray*}
	\end{itemize}
	
\end{thm}

\subsection{The case $G=GL(1|1)$} Let $G=GL(1|1)$, and $\g=\operatorname{Lie}G$.  We give a very explicit description of the $\g\times\g$ action on $\C[G]$.  In this case, there is only one block of $\mathsf{Rep}(G)$ which is not semisimple, the principal block $\BB_{0}$, and it contains the irreducible representations where the center of $\g\l(1|1)$ acts trivially.  We draw a picture depicting the local structure of $M_{\BB_0}$ below. 
 
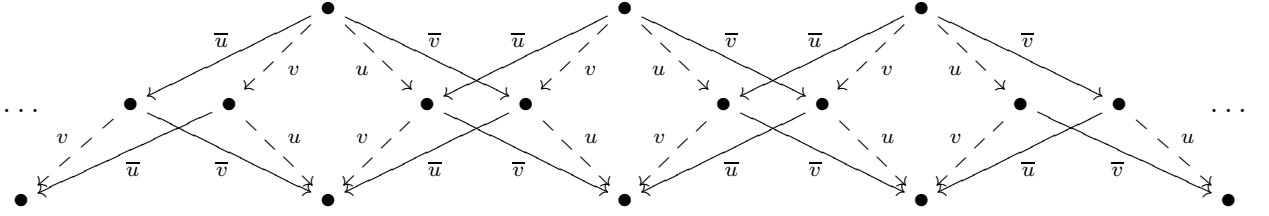
\[
\xymatrix{&&&\bullet \ar[dll]_{\ol{u}}\ar[drr]^{\ol{v}} \ar@{-->}[dl]^{v} \ar@{-->}[dr]_{u}&&&\bullet \ar[dll]_{\ol{u}}\ar[drr]^{\ol{v}} \ar@{-->}[dl]^{v} \ar@{-->}[dr]_{u}&&&\bullet \ar[dll]_{\ol{u}}\ar[drr]^{\ol{v}} \ar@{-->}[dl]^{v} \ar@{-->}[dr]_{u}&&&\\ \hdots&\bullet \ar[drr]_{\ol{v}} \ar@{-->}[dl]_v & \bullet \ar[dll]^{\ol{u}} \ar@{-->}[dr]^u & & \bullet\ar[drr]_{\ol{v}} \ar@{-->}[dl]_v & \bullet \ar[dll]^{\ol{u}} \ar@{-->}[dr]^{u} & &\bullet \ar[drr]_{\ol{v}} \ar@{-->}[dl]_v & \bullet \ar[dll]^{\ol{u}} \ar@{-->}[dr]^u & & \bullet\ar[drr]_{\ol{v}} \ar@{-->}[dl]_{v} & \bullet \ar[dll]^{\ol{u}} \ar@{-->}[dr]^u & \hdots \\\bullet & && \bullet & &&\bullet & & &\bullet & && \bullet & &&\bullet}
\]

Note that $M_{\BB_0}$ is infinite-dimensional since there are infinitely many nonisomorphic simple modules in $\BB_{0}$.  Each dot in the picture represents a weight vector, with the bottom and top rows having even parity and the middle row having odd parity.  We write $u,v$ for the action of the odd weight vectors of $\g\l(1|1)$ by left translation, and $\ol{u},\ol{v}$ for the action of the odd weight vectors by right translation.  One can see rather explicitly here that under left or right translation only this is just a sum of injective modules.

\subsection{Group-graded supergroups and actions}

Next we study spherical varieties for an especially well-understood class of quasireductive supergroups, those which are group-graded.  We first give a definition along with a brief discussion of group-graded supergroups and their actions.  This will closely follow the definitions and theorems of section 4 of \cite{vishnyakova2011complex}, except that we are working in the algebraic category and not the complex analytic category.

Introduce the category $\mathsf{GSV}$ whose objects are supervarieties of the form $X=(|X_0|,\Lambda^\bullet\NN)$ where $X_0$ is a variety and $\NN$ is a coherent sheaf on $X_0$.  In other words the objects are graded supervarieties with a given choice of grading.  This endows all objects of $\mathsf{GSV}$ with a canonical $\Z$-grading on their structure sheaf.  We then define morphisms in this category to be those morphisms of supervarieties that preserve the given $\Z$-gradings.

There is a natural functor $\operatorname{gr}$ from the category of locally graded supervarieties to $\mathsf{GSV}$.  On objects it is given by 
\[
\operatorname{gr}X=(|X|,\bigoplus\limits_{i\geq0}\JJ_X^i/\JJ_X^{i+1}),
\]
so that in particular $|X|=|\operatorname{gr}X|$ and $X(\C)=|\operatorname{gr}X(\C)|$.  Note that the natural map $\Lambda^i \JJ_X/\JJ_X^2\to \JJ_X^i/\JJ_X^{i+1}$ is an isomorphism because of our locally graded assumption.  For a morphism $\psi:X\to Y$ we let $\operatorname{gr}\psi:\operatorname{gr}X\to\operatorname{gr}Y$ be the same map of underlying topological spaces and set 
\[
(\operatorname{gr}\psi)^*:\bigoplus\limits_{i\geq0}\JJ_Y^i/\JJ_Y^{i+1}\to(\operatorname{gr}\psi)_*\bigoplus\limits_{i\geq0}\JJ_X^i/\JJ_X^{i+1}
\]
to be
\[
(\operatorname{gr}\psi)^*(f+\JJ_Y^i)=\psi^*(f)+\JJ_X^i
\]
where $f$ is a section of $\JJ_Y^{i-1}$.

If $X$ and $Y$ are locally graded supervarieties, then $X\times Y$ is a locally graded supervariety in a natural way, and $\JJ_{X\times Y}=p_X^*\JJ_X+p_Y^*\JJ_Y$, where $p_X,p_Y$ are the natural projection maps.  On the other hand, given two graded supervarieties $X'=(|X'|,\Lambda^\bullet \NN_{X'})$, $Y'=(|Y'|,\Lambda^\bullet \NN_{Y'})$, we define their direct product in $\mathsf{GSV}$ to be the direct product of supervarieties $X'\times Y'$ with the natural splitting $\OO_{X'\times Y'}=\Lambda^\bullet(p_{X'_0}^*\NN_{X'}\oplus p_{Y'_0}^*\NN_{Y'})$. Then there is a canonical isomorphism in $\mathsf{GSV}$ $\operatorname{gr}(X\times Y)\cong\operatorname{gr}X\times\operatorname{gr}Y$ coming from the fact that taking tensor product commutes with taking associated graded for filtered vector spaces with finite filtrations. 

If $G$ is an algebraic supergroup, then using the canonical isomorphism $\operatorname{gr}(G\times G)\cong\operatorname{gr}G\times\operatorname{gr}G$ we have that $\operatorname{gr}G$ with the maps $\operatorname{gr}m_G, \operatorname{gr}e_G$ and $\operatorname{gr}s_G$ form an algebraic supergroup.  If $\g=\operatorname{Lie}G$ we write $\g^{\text{gr}}:=\operatorname{Lie}\operatorname{gr}G$.  Further, if $a:G\times X\to X$ is an action of a Lie supergroup on a locally graded supervariety $X$, then $\operatorname{gr}a:\operatorname{gr}(G\times X)\cong\operatorname{gr}G\times\operatorname{gr}X\to\operatorname{gr}X$ defines an action of $\operatorname{gr}G$ on $\operatorname{gr}X$.  

\begin{definition}
	If $G$ is a supergroup, we call $\operatorname{gr}G$ the group-graded supergroup gotten from $G$, and we say $G$ is a group-graded supergroup if $G\cong\operatorname{gr}G$ as supergroups.  If $a:G\times X\to X$ is an action of $G$ on a locally graded supervariety $X$, we call $\operatorname{gr}a$ the graded action of $\operatorname{gr}G$ on $\operatorname{gr}X$, and we say that $a$ is a graded action if it is isomorphic to $\operatorname{gr}a$ in the natural sense. 
\end{definition}

\begin{remark}
	As algebraic supergroups are smooth affine supervarieties, they are always graded.  The property of being group-graded is stronger in that it requires the multiplication and inversion morphisms to respect some grading.
\end{remark}
We give an explicit construction of $\operatorname{gr}G$.  Being affine the supergroup $G$ is graded, so fix a grading of $G$ so that its structure sheaf is equipped with a $\Z$-grading.  We call $G$ with this chosen grading $\operatorname{gr}G$, and we think of it as an object of $\mathsf{GSV}$.  This choice of grading determines a grading of $G\times G$, and thus we may write
\[
m_G^*=\bigoplus\limits_{i\geq 0}(m_G^*)_i, \ \ \ \ \ s_G^*=\bigoplus\limits_{i\geq 0}(s_G^*)_i
\]  
where $(m_G^*)_i$, respectively $(s_G^*)_i$ increase the $\Z$-grading of an element by exactly $i$.  We set $m_{\operatorname{gr}G}^*=(m_G^*)_0$, $s_{\operatorname{gr}G}^*=(s_{G}^*)_0$, and $e_{\operatorname{gr}G}^*=e_G^*$, and these are all algebra homomorphisms.  In this way, the induced maps on the supervariety $G$ given by $m_{\operatorname{gr}G}$, $s_{\operatorname{gr}G}$, and $e_{\operatorname{gr}G}$ become morphisms in $\mathsf{GSV}$ and define the structure of a supergroup on $\operatorname{gr}G$, and thus this supergroup is group-graded.    It follows in particular that we may identify $(\operatorname{gr}G)_0$ and $G_0$ as algebraic groups.

Now since we have constructed $\operatorname{gr}G$ so that it is the same supervariety as $G$ (the only difference being that it has a chosen $\Z$-grading on its structure sheaf), we have an identification $T_eG=T_e\operatorname{gr}G$.  Thus we may canonically identify $\g\cong \g^{\operatorname{gr}}$ as super vector spaces.  Given $u_e\in T_eG$, we write $u_L$ (resp.\ $u_R$) for the corresponding $G$ right-invariant (resp.\ $G$ left-invariant) vector field on $G$, and $\operatorname{gr}u_L$ (resp.\ $\operatorname{gr}u_R$) for the corresponding $\operatorname{gr}G$ right-invariant (resp.\ $\operatorname{gr}G$ left-invariant) vector field on $G$.  Using the $\Z$-grading on $\C[G]$ we may write $u_L=\sum\limits_{i\in\Z}(u_L)_i$ (resp.\ $u_R=\sum\limits_{i\in\Z}(u_R)_i$) as endomorphisms of $k[G]$, where $(u_L)_i$ (resp.\ $(u_R)_i$) changes the $\Z$-grading by $i$.

\begin{lemma}  Let $u_e\in T_e G$.  If $u_e$ is even then $\operatorname{gr}u_L=(u_L)_0$ and $\operatorname{gr}u_R=(u_R)_0$, and if $u_e$ is odd then $\operatorname{gr}u_L=(u_L)_{-1}$ and $\operatorname{gr}u_R=(u_R)_{-1}$.
\end{lemma}
\begin{proof}
	We prove this for right-invariant vectors, with the case of left-invariant vector fields being similar. have
	\[
	u_L=-(u_e\otimes 1)\circ(m_G^*)=\bigoplus\limits_{i\geq 0} -(u_e\otimes 1)\circ(m_G^*)_{i}.
	\]
	For $f\in\C[G]_k$, $(m_G^*)_i(f)\in\bigoplus\limits_{j}\C[G]_{j}\otimes \C[G]_{k+i-j}$.  If $u_e$ is even, then $u_e$ vanishes on $\C[G]_{i}$ for $i>0$, so 
	\[
	-(u_e\otimes 1)\circ(m_G^*)_{i}=(u_L)_{i},
	\]
	so $\operatorname{gr}u_L=(u_L)_0$.  If $u_e$ is odd, then $u_e$ vanishes on $\C[G]_i$ for $i\neq 1$, so
	\[
	-(u_e\otimes 1)\circ(m_G^*)_{i}=(u_L)_{i-1},
	\]
	so $\operatorname{gr}u_L=(u_L)_{-1}$.
\end{proof}

\begin{cor}\label{graded_group_abel_char}
	We have $[\g^{\operatorname{gr}}_{\ol{1}},\g^{\operatorname{gr}}_{\ol{1}}]=0$.  In fact a supergroup $G$ is group-graded if and only if $[\g_{\ol{1}},\g_{\ol{1}}]=0$, where $\g=\operatorname{Lie}G$.
\end{cor}
\begin{proof}
	For the first statement, the supercommutator of two degree (-1)-maps is of degree (-2) with respect to the $\Z$-grading.  However there are no vector fields of degree (-2) on a graded supervariety, thus the supercommutator must be zero.  A proof of the second statement is given in proposition 4.4 of \cite{vishnyakova2011complex}.
\end{proof}
\begin{definition}
	A Lie superalgebra $\g$ is graded if $[\g_{\ol{1}},\g_{\ol{1}}]=0$.
\end{definition}
Now $G_0\times G_0$ acts on $G$ by left and right translation.  Using Koszul's realization of $\C[G]$ as a coinduced algebra on $\C[G_0]$ (see \cite{koszul1982graded}), which gives a natural grading of $G$, we obtain a natural $G_0\times G_0$-equivariant grading (this does not require that $G_0$ is reductive; if $G_0$ is reductive we could also use \cref{equiv_gradings} to find a $G_0\times G_0$-equivariant grading).  Thus if we constructed $\operatorname{gr}G$ as above, then using the $G_0\times G_0$-equivariant grading we would have that if $u_e$ is even, $u_L=(u_L)_0$ and $u_R=(u_R)_0$ since they will preserve the $\Z$-grading.  Thus we have shown:

\begin{lemma}
	If we construct $\operatorname{gr}G$ by use of a $G_0\times G_0$-equivariant grading of $G$, then for an even tangent vector $u_e\in T_eG$, $u_L=\operatorname{gr}u_L$ and $u_R=\operatorname{gr}u_R$.  In particular $\g_{\ol{0}}=\g^{\operatorname{gr}}_{\ol{0}}$ as Lie algebras of vector fields on $G$.  Further, the natural isomorphism of super vector spaces $\g_{\ol{1}}\cong\g^{\operatorname{gr}}_{\ol{1}}$ induced from this grading is an isomorphism of $\g_{\ol{0}}$-modules.
\end{lemma}
\begin{proof}
	It remains to show the second statement.  For this, we observe that for $u\in\g_{\ol{0}}$, $v\in\g_{\ol{1}}$, $[u,v]_i=[u,v_i]$.  Since $\operatorname{gr}v=v_{-1}$, the statement follows.
\end{proof}

We now move on to the study of graded actions.  

\textbf{Assumption}: For the rest of the section we assume that all supervarieties are locally graded.

\begin{lemma}
	Suppose $G$ is a supergroup which acts on a supervariety $X$, and consider the action of $\operatorname{gr}G$ on $\operatorname{gr}X$.  Then for $u\in\g^{\operatorname{gr}}_{\ol{0}}$, $u$ preserves the $\Z$-grading on $\OO_{\operatorname{gr}X}$, and for $u\in\g^{\operatorname{gr}}_{\ol{1}}$, $u$ acts by degree $-1$ on $\OO_{\operatorname{gr}X}$.
\end{lemma}
\begin{proof}
	For $f\in(\OO_{\operatorname{gr}X})_i$, we have
	\[
	u(f)=-(u_e\otimes 1)\circ(\operatorname{gr}a)^*(f).
	\]
	Now since $\operatorname{gr}a$ preserves the $\Z$-grading, we have $(\operatorname{gr}a)^*(f)\in\bigoplus\limits_{0\leq j\leq i}(\OO_{\operatorname{gr}G})_{j}\otimes(\OO_{\operatorname{gr}X})_{i-j}$.  If $u\in\g^{\text{gr}}_{\ol{0}}$, then $u_e$ vanishes on $(\OO_{G})_i$ for $i>0$, and if $u\in\g^{\text{gr}}_{\ol{1}}$ then $u_e$ vanishes on $(\OO_{G})_i$ for $i\neq 1$.  The result follows.
\end{proof}

Now if $K$ is a closed subgroup of $G$ via the inclusion $\phi:K\to G$, then the $\Z$-gradings induced on $\C[G]$ and $\C[K]$ from Koszul's realization make the natural pullback surjection $\phi^*:\C[G]\to \C[K]$ into a graded map. Thus the kernel of this map, $I_K\sub \C[G]$, becomes a graded ideal.  Further, if we consider the group-graded supergroup structure on $K$ and $G$ from these gradings, $\phi$ will be a homomorphism of supergroups $\operatorname{gr}K\to\operatorname{gr}G$.  Thus $\operatorname{gr}\phi=\phi$, and so $I_K=I_{\operatorname{gr}K}$.

\begin{lemma}
	If $X$ is a supervariety and $x\in X(\C)$, then $\operatorname{Stab}_{\operatorname{gr}G}(x)=\operatorname{Stab}_G(x)$ as closed subvarieties of $G$.
\end{lemma}

\begin{proof}
	Write $K=\operatorname{Stab}_G(x)$, $\II_x$ for the maximal ideal sheaf of $x\in X(\C)$ and $\II^{\operatorname{gr}}_x$ for the maximal ideal sheaf of $x\in\operatorname{gr}X(\C)$.  Then by assumption we have $(a_x)^*(\II_x)=I_K$.  But with respect to the $\Z$-grading from Koszul's realization, $I_K$ is a graded ideal and thus $(\operatorname{gr}a_x)^*(\II_{x}^{\operatorname{gr}})=I_K=I_{\operatorname{gr}K}$, and we are done.
\end{proof}

\begin{cor}\label{graded_homog}
	If $X$ is a homogeneous $G$-supervariety isomorphic to $G/K$, then $\operatorname{gr}X$ is a homogeneous $\operatorname{gr}G$-supervariety isomorphic to $\operatorname{gr}G/\operatorname{gr}K$.
\end{cor}

\subsection{$G$ a quasireductive group-graded supergroup}

Let $G$ be a quasireductive supergroup, and write $\g=\operatorname{Lie}G$ as always.  

\begin{lemma}\label{abelian_ideal}
	If $\l\sub\g_{\ol{1}}$ is an abelian ideal of $\g$, then $\l$ is contained in every hyperborel subalgebra of $\g$.
\end{lemma}
\begin{proof}
	If $\b$ is a hyperborel subalgebra, then $\b+\l$ is a subalgebra that still satisfies the first two properties of being a hyperborel, and thus by maximality $\b=\b+\l$.
\end{proof}

\begin{cor}\label{borel_in_graded_case}
	Let $G$ be a group-graded quasireductive supergroup.  Then every hyperborel of $\g$ is of the form $\b_{\ol{0}}\oplus\g_{\ol{1}}$, where $\b_{\ol{0}}$ is a Borel subalgebra of $\g_{\ol{0}}$.  In particular $G$ has only one hyperborel subalgebra 	up to conjugacy. 
\end{cor}
\begin{proof}
	In this case $\g_{\ol{1}}$ is an abelian ideal of $\g$, so we use \cref{abelian_ideal} to get that every hyperborel must contain $\g_{\ol{1}}$, and thus they are all of this form.  If $\b,\b'$ are two hyperborels, then conjugating $\b_{\ol{0}}$ to $\b_{\ol{0}}'$ will conjugate $\b$ to $\b'$.
\end{proof}
In fact, we have
\begin{prop}
	If $\g$ is quasireductive and $\g_{\ol{1}}$ is contained in a hyperborel subalgebra, then $\g$ is graded.
\end{prop}
\begin{proof}
	Since $[\g_{\ol{1}},\g_{\ol{1}}]\sub\g_{\ol{0}}$ is a submodule of the adjoint representation, if it is nonzero it must intersect any Cartan subalgebra nontrivially.  Thus if $\g_{\ol{1}}$ is contained in a hyperborel subalgebra we must have $[\g_{\ol{1}},\g_{\ol{1}}]=0$, i.e.\ $\g$ is graded.
\end{proof}
The following lemma now follows easily from what we have shown so far.
\begin{lemma}\label{borel_graded_sub}
	If $G$ is a quasireductive supergroup, and $B$ is a hyperborel subgroup of $G$, then $\operatorname{gr}G$ is quasireductive and $\operatorname{gr}B$ is a subgroup of a hyperborel of $\operatorname{gr}G$.
\end{lemma}
We can now prove that the functor $\operatorname{gr}$ preserves sphericity.
\begin{cor}
	Suppose that $G$ is quasireductive and $X$ is a spherical $G$-supervariety.  Then $\operatorname{gr}X$ is a locally graded spherical $\operatorname{gr}G$-supervariety under the graded action.
\end{cor}
\begin{proof}
	Let $B$ be a hyperborel of $G$ with an open orbit on $X$.  Then by \cref{graded_homog}, $\operatorname{gr}B$ has an open orbit on the same underlying open subset of $|X|$.  By \cref{borel_graded_sub}, $\operatorname{gr}B$ is contained in a hyperborel of $\operatorname{gr}G$, and the hyperborel of $\operatorname{gr}G$ containing $\operatorname{gr}B$ has an open orbit at $x$.  Thus $\operatorname{gr}X$ is spherical.
\end{proof}

For the rest of this section we assume that $G$ is a group-graded quasireductive supergroup.

\begin{prop}\label{graded_restriction_injective}
	Suppose that $X$ is a locally graded spherical $G$-supervariety.  Then $\operatorname{soc}\C[X]$ is a subalgebra of $\C[X]$ and the restriction of $i_X^*$ to $\operatorname{soc}\C[X]$ is injective.  In particular, $\operatorname{soc}\C[X]$ is an even subalgera of $\C[X]$ without nilpotents.
\end{prop}
\begin{proof}
	A semisimple representation of $G$ is exactly the pullback of a semisimple representation of $G_0$ under the natural surjection $G\to G_0$.  Therefore $\operatorname{soc}\C[X]$ can be thought of as a sum of simple $G_0$-representations, and thus the tensor product of two subrepresentations of $\operatorname{soc}\C[X]$ is again a semisimple $G_0$-representation.  Since multiplication is $G$-equivariant, it follows that $\operatorname{soc}\C[X]$ is a subalgebra of $\C[X]$.  
	
	Recall that $i_X$ is a $G_0$-equivariant map of algebras.  If $\operatorname{soc}\C[X]\cap \ker i_X^*\neq0$, then it must contain a simple subrepresentation $L$.  Let $f\in L$ be the $B$-highest weight vector for some hyperborel $B$ of $G$.  Then by \cref{spher_char}, $f$ is non-nilpotent and thus $i_X^*(f)\neq0$, a contradiction.  This completes the proof.
\end{proof}

\begin{cor}\label{graded_affine_spher_complete_red}
	For a locally graded affine spherical $G$-supervariety $X$, $\C[X]$ is completely reducible if and only if $X=X_0$.
\end{cor}
\begin{proof}
	If $X=X_0$ then $G$ acts via the quotient to $G_0$ so $\C[X]$ is completely reducible.
	
	On the other hand, the condition that $\C[X]$ is completely reducible is equivalent to $\C[X]=\operatorname{soc}\C[X]$.  By \cref{graded_restriction_injective}, this condition implies that $i_X^*$ is an isomorphism, so $X=X_0$.
\end{proof}

We now focus on the case of homogeneous spherical supervarieties for $G$.
\begin{lemma}\label{graded_action_graded_case}
	If $X$ is a homogeneous $G$-supervariety, then $X$ is graded, and the action $a:G\times X\to X$ is isomorphic to the graded action $\operatorname{gr}a$.
\end{lemma}
\begin{proof}
	This follows directly from \cref{graded_homog}.
\end{proof}

\begin{prop}
	If $X$ is a homogeneous $G$-supervariety, then $X$ is spherical if and only if $X_0$ is a spherical $G_0$-variety.
\end{prop}
\begin{proof}
	If $X=G/K$, then we want to determine when $\k=\operatorname{Lie}K$ has a complementary hyperborel subalgebra in $\g=\operatorname{Lie}G$.  By \cref{borel_in_graded_case}, the hyperborels of $\g=\operatorname{Lie}G$ are all of the form $\b_{\ol{0}}\oplus\g_{\ol{1}}$ for a Borel subalgebra $\b_{\ol{0}}$ of $\g_{\ol{0}}$.  Thus it is equivalent to find a Borel subalgebra $\b_{\ol{0}}$ complementary to $\k_{\ol{0}}$ in $\g_{\ol{0}}$.  Since $X_0=G_0/K_0$, this completes the proof.
\end{proof}

\begin{prop}\label{graded_homog_spher}
	If $X$ is a homogeneous spherical $G$-supervariety, then there exists a grading of $X$ for which $\C[X]_{0}=\operatorname{soc}\C[X]$.  In particular, if $B$ is a hyperborel of $G$, then $\Lambda_{B}^+(X)=\Lambda_{B_0}^+(X_0)$.
\end{prop}
\begin{proof}
	By \cref{graded_action_graded_case}, there exists a grading of $X$ for which the action of $G$ is graded.  With respect to this action, $\g_{\ol{1}}$ acts by degree $-1$ derivations on $\OO_{X}$.  Thus $\C[X]_0\sub \C[X]^{\g_{\ol{1}}}=\operatorname{soc} \C[X]$.  On the other hand, by \cref{graded_restriction_injective}, $i_X:\operatorname{soc} \C[X]\to \C[X_0]$ is injective.  Since $i_X: \C[X]_0\to \C[X_0]$ is an isomorphism we must have $ \C[X]_0=\operatorname{soc} \C[X]$.
\end{proof}

In the case of homogeneous affine spaces, we have the following strengthening of \cref{graded_affine_spher_complete_red}.  Note that a homogeneous space $G/K$ is affine if and only if $K_0$ is reductive, i.e.\ $K$ is quasireductive.

\begin{prop}\label{graded_affine_complete_red}
	If $X=G/K$ is a homogeneous affine $G$-space, then the following are equivalent.
	\begin{enumerate}
		\item $X=X_0$.
		\item $ \C[X]$ is completely reducible.
		\item $ \C$ splits of from $ \C[X]$ as a $G$-module.
	\end{enumerate}
\end{prop}
Before proving this, we first state a lemma.
\begin{lemma}\label{graded_affine_ideal_grading}
	Suppose that $G$ is quasireductive and that $\g=\operatorname{Lie}(G)$ has an odd abelian ideal $\l\sub\g_{\ol{1}}$.  Then if $K\sub G$ is a quasireductive subgroup, $ \C$ splits off from $ \C[G/K]$ only if $\l\sub\k=\operatorname{Lie}(K)$.
\end{lemma}
\begin{proof}
	Suppose that $\l$ is not contained in $\k$.  Let $\m=\k\cap\l$, and let $\r$ be a $\k_{\ol{0}}$-invariant complement to $\m$ in $\l$, where we are using that $K_0$ is reductive.  Write $L$,$M$, and $R$ for the purely even vector spaces with $L=\l_{\ol{1}},M=\m_{\ol{1}}$, and $R=\r_{\ol{1}}$.  We may naturally view $L$ as a $\g_{\ol{0}}$-module according to the restriction of the adjoint action of $\g_{\ol{0}}$ to $\l$, using that $\l$ is an ideal of $\g$.
	
	Now consider the following $\g$-module $V$.  As a $\g_{\ol{0}}$-module, $V=L\otimes L^*\oplus\Pi L^*$.  Choose a $\g_{\ol{0}}$-invariant complement $\l'$ to $\l$ in $\g_{\ol{1}}$.  Then we say that for $u\in\l'$, $u$ acts by 0 on $V$, and for $u\in\l$, $u$ acts by $0$ on $V_{0}=L\otimes L^*$, while for $\varphi\in V_{\ol{1}}=\Pi L^*$, we set $u\cdot \varphi:=u\otimes\varphi\in V_{\ol{0}}$.  Then this defines a representation of $\g$ on $V$.  Further, the span of the element $v_L\in V_{\ol{0}}=L\otimes L^*$ which correspond to the identity map on $L$ defines an even trivial subrepresentation $ \C\langle v_L\rangle$ of $V$.  This subrepresentation does not split off of $V$, as we see that if $u_1,\dots,u_n$ is a basis of $L$ and $\varphi_1,\dots,\varphi_n$ is a the parity shift of a dual basis in $\Pi L^*$, then we have the following equation in $V$:
	\[
	\sum\limits_{i=1}^{n}u_i\cdot\varphi_i=\sum\limits_{i=1}^n u_i\otimes\varphi_i=v_L
	\]
	
	Consider the element $\psi\in V^*$ corresponding to the trace form on $R\otimes R^*\sub L\otimes L^*$.  Then as an element of $V^*$, $\psi$ is $\k_{\ol{0}}$-invariant since $R$ is a $\k_{\ol{0}}$-submodule. If $u\in\k_{\ol{1}}$ and $\varphi\in V_{\ol{1}}$, then $u\cdot\varphi=u\otimes\varphi\in M\otimes L^*$, and thus $\psi(u\otimes \varphi)=0$.  It follows that $\psi\in(V^*)^{\k}$, i.e.\ it defines an even coinvariant of $V$, so by Frobenius reciprocity it defines a $G$-module morphism $\Psi:V\to \C[G/K]$.  Further, since $\psi(v_L)\neq0$ and $v_L$ is $G$-fixed, $\Psi(v_L)$ is a non-zero constant function on $G/K$.  We see that
	\[
	\sum\limits_{i=1}^n u_i\cdot\Psi(\varphi_i)=\Psi\left(\sum\limits_{i=1}^{n}u_i\cdot\varphi_i\right)=\Psi(v_L).
	\]
	It follows that $\C$ does not split off from $ \C[G/K]$, and we are done.
\end{proof}

Now we prove \cref{graded_affine_complete_red}.

\begin{proof}
	Since $\g_{\ol{1}}$ if an odd abelian ideal of $\g$, if $K\sub G$ is a quasireductive subgroup, $\C$ splits off from $ \C[G/K]$ only if $\g_{\ol{1}}\sub\k$ by \cref{graded_affine_ideal_grading}, and in this case $G/K$ is a purely even variety.  This shows $(3)\implies(1)$.  Both $(1)\implies(2)$ and $(2)\implies(3)$ are obvious.
\end{proof}

\section{Appendix: Action of Lie Superalgebras}\label{app_superalg_actions}
Let $\g$ be an arbitrary Lie superalgebra and $X$ a supervariety.  An action of $\g$ on $X$ is a Lie superalgebra homomorphism $\g\to\Gamma(X,\TT_X)$.

\begin{definition}\label{open_alg_orbit}
	If $\g$ acts on $X$, then we say $\g$ has an open orbit on $X$ if there exists a point $x\in X(\C)$ such that the natural restriction map $\g\to T_xX$ is a surjection.  In this case, the locus of points where $\g\to T_xX$ is surjective is open, and we call this open set an open orbit of $\g$.  We say $X$ is a homogeneous $\g$-supervariety if all of $X$ is an open orbit.
\end{definition}

\begin{prop}\label{ideal_orbit_alg}
	Suppose that $X$ is a homogeneous $\g$-supervariety.  If $\LL$ is a $\g$-equivariant line bundle on $X$, then a nonzero $\g$-submodule of $\Gamma(X,\LL)$ generates $\LL$. In particular, if $X$ is affine, $\C[X]$ has no non-trivial $\g$-invariant ideals.
\end{prop}
\begin{proof}
	See \cref{algebra_action_remark}.
\end{proof}
Now assume that $\g$ is quasireductive.
\begin{definition}
		A $\g$-supervariety $X$ is said to be spherical if there exists a hyperborel subalgebra $\b$ in $\g$ such that $\b$ has an open orbit on $X$.  In this case we say that $X$ is $\b$-spherical.
\end{definition}
\begin{remark}If $G$ is quasireductive and acts on a supervariety $X$, and $B$ is a hyperborel subgroup of $G$, then $X$ is $B$-spherical if and only if $X$ is $\b$-spherical for the induced action of $\g$ on $X$.
\end{remark}
We may now slightly extend our results on spherical supervarieties.
\begin{thm}\label{alg_action_spher}
	Let $X$ be a $\g$-supervariety, $\b$ a hyperborel subalgebra of $\g$ and $\h_{\ol{0}}\sub\b$ a Cartan subalgebra of $\g_{\ol{0}}$.  If $X$ is $\b$-spherical then for a $\b$-equivariant line bundle $\LL$ on $X$, $\Gamma(X,\LL)^{(\b)}$ is a multiplicity-free $\h_{\ol{0}}$-module and if $s\in\Gamma(X,\LL)^{(\b)}$ is non-zero then it is non-vanishing.
\end{thm}
\begin{proof}
	Suppose that $s\in \Gamma(X,\LL)^{(\b)}$ is a non-zero weight vector of $\b$. If we restrict $s$ to the open orbit $U$ of $\b$, by \cref{ideal_orbit_alg} it must generate $\LL|_U$ since it cannot restrict to zero.  This implies the restriction of $s$ to $U$  must be non-vanishing.
	
	Now if $s_1,s_2\in\Gamma(X,\LL)^{(\b)}$ are non-zero weight vectors for $\b$ of the same weight, then $f=s_1/s_2$ is a rational $\b$-invariant function.  Since $s_2$ is non-vanishing on $U$, $f$ is regular on $U$.  We may assume by further restriction that $U$ affine.  Then since it is $\b$-homogeneous, $\C[U]$ has no nontrivial $\b$-invariant ideals by \cref{ideal_orbit_alg}.  However for $x\in U(\C)$, $(f-f(x))$ will be an invariant ideal which is not equal to $ \C[U]$ since it is contained in $\m_x$.  Therefore $f-f(x)=0$, so $f$ is constant, and thus $s_1$ and $s_2$ are proportional.  This completes the proof.
\end{proof}

\section{Appendix: Smoothness}\label{app_smooth}

Let $X$ be a complex supervariety and let $x\in X(\C)$. We say that $X$ is smooth at $x$ if the natural evaluation map $\TT_{X,x}\to T_xX$ is surjective (see \cref{smooth_rmk}).  We seek to give a list of conditions that are equivalent to this, so as to clarify the existing literature on smoothness of superschemes.  In order to state the characterization, we recall several constructions and definitions.

\begin{itemize}
	\item For a supervariety $X$, write $\Omega_X$ for its sheaf of differentials, which can be defined as the conormal sheaf to $X$ under the diagonal embedding $X\to X\times X$.
	\item For $x\in X(\C)$ we may view $T_xX$ as the affine superspace $\Spec S^\bullet(\m_x/\m_x^2)$.  Define the tangent cone at $x$, $TC_xX$, to be the closed conical subvariety of $T_xX$ given by
	\[
	TC_xX=\Spec\left(\bigoplus\limits_{n\geq 0}\m_x^n/\m_x^{n+1}\right)
	\]
	The derivations in $\TT_{X,x}$ act on both $\C[T_xX]$ and $\C[TC_x]$ by derivations of degree -1, and the action is equivariant with respect to the above closed embedding.  
	\item For a local supercommutative algebra $A$ with unique maximal ideal $\m$, we write $\widehat{A}$ for the completion of $A$ with respect to the $\m$-adic topology. 
	\item Following \cite{schmitt1989regular}, given a superalgebra $A$ we say that an even element $t\in A_{\ol{0}}$ is $A$-regular if the multiplication map by $t$ is injective.  We say an odd element $\xi\in A_{\ol{1}}$ is $A$-regular if the cohomology of the multiplication map by $\xi$ is trivial.  Finally, if $(r_1,\dots,r_k)$ is a sequence of homogeneous elements of $A$, we say the sequence is $A$-regular if $r_i$ is regular in $A/(r_1,\dots,r_{i-1})$.  
\end{itemize}
\begin{definition}
	A local supercommutative algebra $A$ is regular if the unique maximal ideal $\m$ is generated by an $A$-regular sequence.
\end{definition}

\begin{lemma}\label{generic freeness}
	Let $\ol{F}$ be a finitely generated field over $\C$ of transcendence degree $m$, and let $F=\ol{F}[\xi_1,\dots,\xi_n]$ for odd variables $\xi_1,\dots,\xi_n$.  Then $\Omega_{F/\C}$ is a free $F$-module of rank $(m|n)$.  
\end{lemma}

\begin{proof}
	We have the short exact sequence
	\[
	F\otimes_{\ol{F}}\Omega_{\ol{F}/\C}\to\Omega_{F/\C}\to\Omega_{F/\ol{F}}\to0.
	\]
	Since $\Omega_{F/\ol{F}}$ is a free $F$-module of rank $(0|n)$ with generators $d\xi_1,\dots,d\xi_n$, the last map splits which implies that $d\xi_1,\dots,d\xi_n$ generate a free summand of $\Omega_{F/\C}$ of rank $(0|n)$.  We know that $\Omega_{\ol{F}/\C}$ is a free $\ol{F}$-module of rank $(m|0)$ with generators $dt_1,\dots,dt_m$, where $t_1,\dots,t_m$ form a transcendence basis of $\ol{F}$ over $\C$.  Hence $\Omega_{F/\C}$ is generated by $dt_1,\dots,dt_m,d\xi_1,\dots,d\xi_n$, and it suffices to show that $dt_1,\dots,dt_m$ are $F$-linearly independent.
	
	However if we compute $\ul{Hom}_F(\Omega_{F/\C},F)$ we get the module of $\C$-linear derivations of $F$, which contains a free submodule of rank $(m|0)$ generated by $\d_{t_1},\dots,\d_{t_m}$.  These may be used to show that $dt_1,\dots,dt_m$ are $F$-linearly independent, and we are done.
	
\end{proof}

\begin{thm}\label{smoothness} For a supervariety $X$ and closed point $x\in X(\C)$, let $A:=\OO_{X,x}$ with maximal ideal $\m=\m_x$.  Let $t_1,\dots,t_m$, $\xi_1,\dots,\xi_n\in\m$  project to a homogeneous basis of $\m/\m^2$, where $\ol{t_i}=\ol{0}$ and $\ol{\xi_i}=\ol{1}$.  Then the following are equivalent.
	\begin{enumerate}
		\item $\widehat{A}\cong \C\llbracket t_1,\dots,t_m,\xi_1,\dots,\xi_n\rrbracket$;
		\item $Gr_{\m}A:=\bigoplus\limits_{n\geq 0}\m^n/\m^{n+1}\cong \C[\ul{t_1},\dots,\ul{t_m},\ul{\xi_1},\dots,\ul{\xi_n}]$, where $\ul{(\cdot)}:\m\to\m/\m^2$ is the natural projection;
		\item $\Omega_{X,x}=\Omega_{A/\C}$ is free over $A$;
		\item $\Spec A\to \C$ is a formally smooth morphism;
		\item $\ol{A}=A/(A_{\ol{1}})$ is a regular local ring, and $A\cong \ol{A}[\xi_1,\dots,\xi_n]$;
		\item there exists an affine neighborhood $U=\Spec B$ of $x$ such that $\ol{B}=B/(B_{\ol{1}})$ is regular and $B\cong \Lambda^{\bullet}\ol{B}^{\oplus n}$;
		\item $T_xX=TC_xX$;
		\item the natural map $\TT_{X,x}\to T_xX$ is surjective;
		\item $A$ is a regular local superalgebra;
		\item $A$ is a graded integral superdomain such that $\ol{A}$ is a regular local ring;
	\end{enumerate}
\end{thm}

\begin{proof}	
	The equivalence $(1)\iff(2)$ is proven in \cite{fioresi2008smoothness}, $(2)\iff(7)$ is clear, $(3)\iff(4)$ is proven in \cite{kapranov2011supersymmetry}, and $(5)\iff(9)$ is proven in \cite{schmitt1989regular}.  The equivalence (5)$\iff$(10) follows from the following commutative algebra statement (which is proven in the author's PhD thesis): if $M$ is a finitely generated module over a local Noetherian integral domain, then $M$ is free if and only if every exterior power of $M$ is torsion-free.
	
	For $(1)\implies (3)$, we have that $\m/\m^2\cong\widehat{\m}/\widehat{\m}^2$ is $(m|n)$-dimensional, so by Nakayama's lemma $\Omega_{A/\C}$ is generated by $(m|n)$ elements.  Localizing $A$ at the generic point, we obtain a superalgebra $F$ which by our assumptions and the Cohen structure theorem is isomorphic to $\ol{F}[\xi_1,\dots,\xi_n]$, where $\ol{F}$ is the fraction field of $\ol{A}$.  Hence by \cref{generic freeness} $\Omega_{F/\C}$, which is the localization of $\Omega_{A/\C}$ at the generic point, is free of rank $(m|n)$.  It follows that $\Omega_{A/\C}$ must itself be free of rank $(m|n)$.
	
	For $(3)\implies(8)$, $dt_1,\dots,dt_m,d\xi_1,\dots,d\xi_n$ form a basis of $\Omega_{A/\C}$.  Then 
	\[
	\TT_{X,x}=\ul{\Hom}_A(\Omega_{A/\C},A)
	\]
	will be free with basis $\d_{t_1},\dots,\d_{t_m},\d_{\xi_1},\dots,\d_{\xi_n}$ and these derivations map to a basis of $T_xX$, namely the dual basis of $\ul{t_1},\dots,\ul{t_m},$ $\ul{\xi_1},\dots,\ul{\xi_n}\in\m/\m^2$.
	
	$(8)\implies (7)$: If $TC_xX\neq T_xX$, then the vanishing ideal of $TC_xX$ must be preserved by all derivations from $\TT_{X,x}$.  By our assumption, we get all coordinate derivations from the derivations of $\TT_{X,x}$, so no such non-trivial ideals exist. 
	
	For $(5)\iff(6)$, the backward direction follows from localizing. Conversely, the isomorphism $\OO_{X,x}\to\ol{A}[\xi_1,\dots,\xi_n]$ may be extended to a morphism of sheaves $\OO_X\to\operatorname{gr}X$ on a small enough affine open of $x$.  This morphism will be an isomorphism of stalks at $x$, and so using Noetherian and coherent properties, we get an isomorphism in an open neighborhood of $x$.
	
	The implication $(5)\implies(1)$ is clear.  
	
	Now we assume $(1)$, and use $(3)$ (which we have so far shown is equivalent to (1)) to prove $(5)$.  First, (1) implies that $\ol{A}$ is regular.  As noted previously, by (3) we know that $A$ has derivations $\d_{t_1},\dots,\d_{t_m},\d_{\xi_1},\dots,\d_{\xi_n}$.  These derivations extend canonically to $\widehat{A}$ as the usual coordinate derivations, and these derivations preserve $A$ as a subalgebra.  We have the following diagram:
	\[
	\xymatrix{A\ar[r]\ar[d]^{\pi} & \C\llbracket t_1,\dots,t_m,\xi_1,\dots,\xi_n\rrbracket\ar[d]^{\widehat{\pi}}\\
		\ol{A} \ar[r] & \C\llbracket t_1,\dots,t_m\rrbracket}
	\]
	where $\pi$ is the natural quotient map.  To construct a splitting $\ol{A}\to A$, we observe that $\widehat{\pi}$ has a natural splitting $\widehat{s}$ sending $t_i$ to $t_i$.   We would like to show that $\widehat{s}(\ol{A})$ lies in the image of $A$ in the completion.
	
	Let $f\in \ol{A}$, thought of as a power series.  Then we may lift $f$ to $\widetilde{f}\in A_{\ol{0}}$.  The power series expansion of $\widetilde{f}$ will then be 
	\[
	\widetilde{f}=f+\sum\limits_{I\neq\emptyset} f_I\xi_I\in A
	\]
	where $\xi_I=\xi_{i_1}\cdots\xi_{i_k}$ if $I=\{i_1,\dots,i_k\}$, and $f_I\in \C\llbracket t_1,\dots,t_m\rrbracket$.  
	Using the derivations $\d_{\xi_i}$ for varying $I$, we may show that each function $f_I$ lies in $A$, and so $f$ itself lies in $A$.  Therefore we have our splitting, and now it follows that $A\cong\ol{A}[\xi_1,\dots,\xi_n]$.  
\end{proof}

\bibliographystyle{amsalpha}
\bibliography{bibliography}

\providecommand{\bysame}{\leavevmode\hbox to3em{\hrulefill}\thinspace}
\providecommand{\MR}{\relax\ifhmode\unskip\space\fi MR }
\providecommand{\MRhref}[2]{%
  \href{http://www.ams.org/mathscinet-getitem?mr=#1}{#2}
}
\providecommand{\href}[2]{#2}
\begin{thebibliography}{KKLV89}

\bibitem[All12]{alldridge2012harish}
Alexander Alldridge, \emph{The {Harish-Chandra} isomorphism for reductive
  symmetric superpairs}, Transformation Groups \textbf{17} (2012), no.~4,
  889--919.

\bibitem[AS15]{alldridge2015spherical}
Alexander Alldridge and Sebastian Schmittner, \emph{Spherical representations
  of {Lie} supergroups}, Journal of Functional Analysis \textbf{268} (2015),
  no.~6, 1403--1453.

\bibitem[Bal11]{balduzzi2011supermanifolds}
L~Balduzzi, \emph{Supermanifolds, super {Lie} groups, and super
  {Harish-Chandra} pairs functorial methods and actions}, Ph.D. thesis, PhD
  thesis, Universita degli Studi di Genova, 2011.

\bibitem[CCF11]{carmeli2011mathematical}
Claudio Carmeli, Lauren Caston, and Rita Fioresi, \emph{Mathematical
  foundations of supersymmetry}, vol.~15, European Mathematical Society, 2011.

\bibitem[CN18]{cacciatori2018projective}
Sergio~Luigi Cacciatori and Simone Noja, \emph{Projective superspaces in
  practice}, Journal of Geometry and Physics \textbf{130} (2018), 40--62.

\bibitem[CW12]{cheng2012dualities}
Shun-Jen Cheng and Weiqiang Wang, \emph{Dualities and representations of {Lie}
  superalgebras}, American Mathematical Soc., 2012.

\bibitem[DM99]{deligne1999notes}
Pierre Deligne and John Morgan, \emph{Notes on supersymmetry (following {Joseph
  Bernstein})}, Quantum fields and strings: a course for mathematicians (1999),
  41--97.

\bibitem[Fio08]{fioresi2008smoothness}
Rita Fioresi, \emph{Smoothness of algebraic supervarieties and supergroups},
  Pacific Journal of Mathematics \textbf{234} (2008), no.~2, 295--310.

\bibitem[Gre76]{green1976locally}
JA~Green, \emph{Locally finite representations}, Journal of Algebra \textbf{41}
  (1976), no.~1, 137--171.

\bibitem[Har13]{hartshorne2013algebraic}
Robin Hartshorne, \emph{Algebraic geometry}, vol.~52, Springer Science \&
  Business Media, 2013.

\bibitem[Kac80]{kac1980some}
Victor~G Kac, \emph{Some remarks on nilpotent orbits}, J. algebra \textbf{64}
  (1980), 190--213.

\bibitem[KKLV89]{knop1989local}
Friedrich Knop, Hanspeter Kraft, Domingo Luna, and Thierry Vust, \emph{Local
  properties of algebraic group actions}, Algebraische Transformationsgruppen
  und Invariantentheorie Algebraic Transformation Groups and Invariant Theory,
  Springer, 1989, pp.~63--75.

\bibitem[Kos82]{koszul1982graded}
Jean-Louis Koszul, \emph{Graded manifolds and graded {Lie} algebras},
  Proceedings of the International Meeting on Geometry and Physics (Bologna),
  Pitagora, 1982, pp.~71--84.

\bibitem[Kos94]{koszul1994connections}
\bysame, \emph{Connections and splittings of supermanifolds}, Differential
  Geometry and its Applications \textbf{4} (1994), no.~2, 151--161.

\bibitem[KV11]{kapranov2011supersymmetry}
Mikhail Kapranov and Eric Vasserot, \emph{Supersymmetry and the formal loop
  space}, Advances in Mathematics \textbf{227} (2011), no.~3, 1078--1128.

\bibitem[Los09]{losev2009proof}
Ivan~V Losev, \emph{Proof of the {Knop} conjecture}, Annales de l'Institut
  Fourier, vol.~59, 2009, pp.~1105--1134.

\bibitem[LSZ12]{la2012donkin}
Roberto La~Scala and Alexandr~N Zubkov, \emph{{Donkin--Koppinen} filtration for
  general linear supergroups}, Algebras and representation theory \textbf{15}
  (2012), no.~5, 883--899.

\bibitem[Man13]{manin2013gauge}
Yuri~I Manin, \emph{Gauge field theory and complex geometry}, vol. 289,
  Springer Science \& Business Media, 2013.

\bibitem[Mus12]{musson2012lie}
Ian~Malcolm Musson, \emph{{Lie} superalgebras and enveloping algebras}, vol.
  131, American Mathematical Soc., 2012.

\bibitem[PS94]{penkov1994generic}
Ivan Penkov and Vera Serganova, \emph{Generic irreducible representations of
  finite-dimensional {Lie} superalgebras}, International Journal of Mathematics
  \textbf{5} (1994), no.~03, 389--419.

\bibitem[Rot93]{rothstein1993equivariant}
Mitchell Rothstein, \emph{Equivariant splittings of supermanifolds}, Journal of
  Geometry and Physics \textbf{12} (1993), no.~2, 145--152.

\bibitem[Sch87]{scheunert1987invariant}
Manfred Scheunert, \emph{Invariant supersymmetric multilinear forms and the
  casimir elements of p-type {Lie} superalgebras}, Journal of mathematical
  physics \textbf{28} (1987), no.~5, 1180--1191.

\bibitem[Sch89]{schmitt1989regular}
Thomas Schmitt, \emph{Regular sequences in {Z2}-graded commutative algebra},
  Journal of Algebra \textbf{124} (1989), no.~1, 60--118.

\bibitem[Ser83]{serganova1983classification}
Vera~V Serganova, \emph{Classification of real simple lie superalgebras and
  symmetric superspaces}, Functional Analysis and Its Applications \textbf{17}
  (1983), no.~3, 200--207.

\bibitem[Ser96]{serganova1996generalizations}
Vera Serganova, \emph{On generalizations of root systems}, Communications in
  Algebra \textbf{24} (1996), no.~13, 4281--4299.

\bibitem[Ser11]{serganova2011quasireductive}
\bysame, \emph{Quasireductive supergroups}, New developments in {Lie} theory
  and its applications \textbf{544} (2011), 141--159.

\bibitem[Ser18]{serganova2018representations}
\bysame, \emph{Representations of a central extension of the simple {Lie}
  superalgebra $\mathfrak{p}(3)$}, S{\~a}o Paulo Journal of Mathematical
  Sciences \textbf{12} (2018), no.~2, 359--376.

\bibitem[She20a]{alex2020iwasawa}
Alexander Sherman, \emph{Iwasawa decomposition for lie superalgebras}, 2020.

\bibitem[She20b]{sherman2020layers}
Alexander Sherman, \emph{Layers of the coradical filtration}, arXiv preprint
  arXiv:2004.00657 (2020).

\bibitem[She20c]{sherman2020spherical}
\bysame, \emph{Spherical indecomposable representations of lie superalgebras},
  Journal of Algebra \textbf{547} (2020), 262--311.

\bibitem[SS16]{sahi2016capelli}
Siddhartha Sahi and Hadi Salmasian, \emph{The {Capelli} problem for
  $\mathfrak{gl}(m|n)$ and the spectrum of invariant differential operators},
  Advances in Mathematics \textbf{303} (2016), 1--38.

\bibitem[SSS18]{sahi2018capelli}
Siddhartha Sahi, Hadi Salmasian, and Vera Serganova, \emph{The {Capelli}
  eigenvalue problem for {Lie} superalgebras}, arXiv preprint arXiv:1807.07340
  (2018).

\bibitem[SV04]{sergeev2004deformed}
AN~Sergeev and AP~Veselov, \emph{Deformed quantum calogero-moser problems and
  lie superalgebras}, Communications in mathematical physics \textbf{245}
  (2004), no.~2, 249--278.

\bibitem[Tim11]{timashev2011homogeneous}
Dmitry~A Timashev, \emph{Homogeneous spaces and equivariant embeddings}, vol.
  138, Springer Science \& Business Media, 2011.

\bibitem[Vis11]{vishnyakova2011complex}
EG~Vishnyakova, \emph{On complex {Lie} supergroups and split homogeneous
  supermanifolds}, Transformation groups \textbf{16} (2011), no.~1, 265--285.

\bibitem[VMP90]{voronov1990elements}
AA~Voronov, Yu~I Manin, and IB~Penkov, \emph{Elements of supergeometry},
  Journal of Soviet Mathematics \textbf{51} (1990), no.~1, 2069--2083.

\end{thebibliography}

\textsc{\footnotesize Dept. of Mathematics, Ben Gurion University, Beer-Sheva,	Israel} 

\textit{\footnotesize Email address:} \texttt{\footnotesize xandersherm@gmail.com}

\end{document}